\documentclass{amsart} 
\usepackage{amssymb,latexsym,amsfonts,verbatim,amscd,ifthen,amsmath, graphicx}
\usepackage{color,mathtools,tikz}

\numberwithin{equation}{section}

\theoremstyle{plain}

\newtheorem{theorem}[equation]{Theorem}   
\newtheorem{lemma}[equation]{Lemma} 
\newtheorem{proposition}[equation]{Proposition}

\theoremstyle{definition}

\newtheorem{definition}[equation]{Definition} 
\newtheorem{remark}[equation]{Remark} 
\newtheorem{remarks}[equation]{Remarks} 
 
\newtheorem{example}[equation]{Example}
\newtheorem{examples}[equation]{Examples}

\newtheorem*{definition*}{Definition}
\newtheorem*{acknowledgments}{Acknowledgments}

\makeatletter
\newsavebox{\@brx}
\newcommand{\llangle}[1][]{\savebox{\@brx}{\(\m@th{#1\langle}\)}%
  \mathopen{\copy\@brx\mkern2mu\kern-0.9\wd\@brx\usebox{\@brx}}}
\newcommand{\rrangle}[1][]{\savebox{\@brx}{\(\m@th{#1\rangle}\)}%
  \mathclose{\copy\@brx\mkern2mu\kern-0.9\wd\@brx\usebox{\@brx}}}
\makeatother

\makeatletter
\newsavebox{\@brxa}
\newsavebox{\@brxb}
\newcommand{\lbrangle}[1][]{\savebox{\@brxa}{\(\m@th{#1(}\)}%
                           \savebox{\@brxb}{\(\m@th{#1(}\)}%
  \mathopen{\copy\@brxa\mkern2mu\kern-0.9\wd\@brxb\usebox{\@brxb}}}
\newcommand{\rbrangle}[1][]{\savebox{\@brxa}{\(\m@th{#1\rbrace}\)}%
                            \savebox{\@brxb}{\(\m@th{#1\rbrace}\)}%
  \mathclose{\copy\@brxb\mkern2mu\kern-0.9\wd\@brxb\usebox{\@brxa}}}
\makeatother

\DeclareMathOperator{\Hom}{Hom} 
\DeclareMathOperator{\Ext}{Ext} 
 
\DeclareMathOperator{\HH}{H} 
\DeclareMathOperator{\RH}{\widetilde H}

\DeclareMathOperator{\id}{id}  
\DeclareMathOperator{\coker}{Coker}

\DeclareMathOperator{\rank}{rank}
\DeclareMathOperator{\im}{Im} 
\DeclareMathOperator{\pd}{pd}

\DeclareMathOperator{\Ker}{Ker} 
\DeclareMathOperator{\supp}{supp}
\DeclareMathOperator{\sign}{sign}

\begin{document}   

\renewcommand{\:}{\! :} 
\newcommand{\p}{\mathfrak p} 
\newcommand{\m}{\mathfrak m}
\newcommand{\e}{\epsilon}
\newcommand{\lra}{\longrightarrow}
\newcommand{\lla}{\longleftarrow}
\newcommand{\la}{\leftarrow} 
\newcommand{\ra}{\rightarrow} 
\newcommand{\altref}[1]{{\upshape(\ref{#1})}} 
\newcommand{\bfa}{\boldsymbol{\alpha}} 
\newcommand{\bfb}{\boldsymbol{\beta}} 
\newcommand{\bfg}{\boldsymbol{\gamma}} 
\newcommand{\bfM}{\mathbf M} 
\newcommand{\bfI}{\mathbf I} 
\newcommand{\bfC}{\mathbf C} 
\newcommand{\bfB}{\mathbf B} 
\newcommand{\bsfC}{\bold{\mathsf C}} 
\newcommand{\bsfT}{\bold{\mathsf T}}
\newcommand{\smsm}{\smallsetminus} 
\newcommand{\ol}{\overline} 
\newcommand{\mgf}{\, \Bbbk} 
\newcommand{\mbf}[1]{\mathbf{#1}}
\newcommand{\ds}{\displaystyle} 
\newcommand{\mdeg}{{\textup{mdeg}}}

\newlength{\wdtha}
\newlength{\wdthb}
\newlength{\wdthc}
\newlength{\wdthd}
\newcommand{\elabel}[1]
           {\label{#1}  
            \setlength{\wdtha}{.4\marginparwidth}
            \settowidth{\wdthb}{\tt\small{#1}} 
            \addtolength{\wdthb}{\wdtha}
            \raisebox{\baselineskip}
            {\color{red} 
             \hspace*{-\wdthb}\tt\small{#1}\hspace{\wdtha}}}  

\newcommand{\mlabel}[1] 
           {\label{#1} 
            \setlength{\wdtha}{\textwidth}
            \setlength{\wdthb}{\wdtha} 
            \addtolength{\wdthb}{\marginparsep} 
            \addtolength{\wdthb}{\marginparwidth}
            \setlength{\wdthc}{\marginparwidth}
            \setlength{\wdthd}{\marginparsep}
            \addtolength{\wdtha}{2\wdthc}
            \addtolength{\wdtha}{2\marginparsep} 
            \setlength{\marginparwidth}{\wdtha}
            \setlength{\marginparsep}{-\wdthb} 
            \setlength{\wdtha}{\wdthc} 
            \addtolength{\wdtha}{1.4ex} 
            \settowidth{\wdthb}{\tt\small{#1}} 
            \marginpar{\vspace*{\baselineskip}
\smash{\raisebox{0.7\baselineskip}{\tt\small{#1}}\hspace{-\wdthb}%
\raisebox{.3\baselineskip}
            {\rule{\wdtha}{0.5pt}}} }
            \setlength{\marginparwidth}{\wdthc} 
            \setlength{\marginparsep}{\wdthd}  }  
            
\renewcommand{\mlabel}{\label}
\renewcommand{\elabel}{\label}


\title[Minimal resolutions are supported on posets]
{Minimal free resolutions 
of monomial ideals and of toric rings 
are supported on posets}

\author{Timothy B. P. Clark}
\address{Department of Mathematics and Statistics\\ 
         Loyola University Maryland \\ 
         4501 North Charles Street \\
         Baltimore, MD 21210}
\email{tbclark@loyola.edu}

\author{Alexandre B. Tchernev} 
\address{Department of Mathematics and Statistics\\
         University at Albany, SUNY\\
         1400 Washington Avenue\\  
         Albany, NY 12222}
\email{atchernev@albany.edu}
\begin{abstract} 
We introduce the notion of a \emph{resolution supported on 
a poset}. When the poset is a CW-poset, i.e. the face poset 
of a regular CW-complex, we recover the 
notion of cellular resolution as introduced by Bayer and Sturmfels. 
Work of Reiner and Welker, and of Velasco,  
has shown that there are monomial ideals 
whose minimal free resolutions are not cellular, hence  
cannot be supported on any CW-poset. We show that for any 
monomial ideal there is a \emph{homology CW-poset} 
that supports a minimal free resolution of the ideal. 
This allows one to extend to every minimal resolution, essentially 
verbatim, techniques initially developed to study cellular 
resolutions. As two demonstrations of this process, 
we show that minimal resolutions of 
toric rings are supported on what we call toric hcw-posets, and  
we give a new combinatorial proof of a fundamental result of Miller  
on the relationship between Artininizations and 
Alexander duality of monomial ideals. 
\end{abstract} 

\maketitle 

\section*{Introduction}

Monomial ideals provide a fundamental gateway for interactions 
between commutative algebra, combinatorics, and algebraic geometry. 
An important question arising in all these areas, and subject to 
intensive past and current research, see e.g. \cite{MiSt, HeHi, Yuz, GeHaMiNa} 
and the references there,  is how to access and control better the 
homological properties and the specializations of a given monomial ideal $I$.   
The lack of  
understanding, in general, of the structure of the maps in a minimal free 
resolution of $I$, provides a significant obstacle. One of 
the most successful approaches to this problem, building on 
the works of Taylor \cite{Tay} and Bayer, Peeva, and Sturmfels~\cite{BaPeSt}, 
is the use 
of \emph{cellular resolutions}, a notion introduced  by 
Bayer and Sturmfels~\cite{BaSt}, and subsequently generalized by 
Batzies and Welker~\cite{BaWe}. With our applications in mind,  in this 
paper we make an effort to keep terminology as consistent as 
possible with that of Miller and Sturmfels~\cite{MiSt}, 
and we reserve the term 
\emph{cellular resolution} for the original notion from \cite{BaSt}, 
while, when needed, we use the term \emph{CW-resolution} for the more 
general notion from \cite{BaWe}.   

Given a cellular resolution of $I$, its  differentials 
are encoded in 
the boundary maps of the cellular chain complex of a 
regular CW-complex. The presence of such a cellular structure     
allows for the introduction of powerful methods from topological 
combinatorics to the study of homological properties of $I$. 
Even 
when the cellular resolution is not minimal, the construction and 
variations on it have  interesting applications in geometry, cf. e.g.  
\cite{BaSt, JoMi, DeYu, Yuz}, lead to formulas for homological 
invariants of $I$ in terms of combinatorial and  topological invariants 
of the regular CW-complex; and, in many cases, lead to a complete 
description of the minimal resolution of $I$ either directly, or 
through methods such as algebraic discrete Morse theory 
\cite{BaWe, JoWe, Sko, OrWe}. 
When the cellular resolution is minimal then there is also a subtle  
connection between the geometric properties of $I$ and the topology 
of the corresponding regular CW-complex. (A very instructive recent 
example of this phenomenon 
is the result of Okazaki and Yanagawa \cite{OkYa} that when $I$ 
is a stable ideal and $X$ is the regular 
CW-complex for the cellular structure described by Clark~\cite{Cl2} 
on a minimal resolution of $I$, then 
Cohen-Macaulayness of $I$ implies that $X$  is homeomorphic to a closed ball.) 
Just as importantly, even for minimal resolutions that are 
well-known, a cellular structure 
provides otherwise unavailable 
effective tools to study linearity properties of 
the resolution, as well as duality;  
see \cite[Sections 5.3 and 5.4]{MiSt}. 

There has been a significant amount 
of current research identifying cellular structures for minimal resolutions 
coming from geometry, combinatorics, and algebra, see e.g. 
\cite{BaPoSt, BlYu, BrBrKl, Cl2, CoNa1, CoNa2, DeYu, DoEn, DoJoSa, DoMo, 
DoSa, Flo1, KuKu, Me, NaRe, No, OkYa, Si}. 
%
%
However, there exist minimal resolutions that do not have a cellular 
structure, as shown by Reiner and Welker~\cite{ReWe}.  
Furthermore, 
Velasco~\cite{Ve} has shown that there are minimal free 
resolutions that cannot be CW-resolutions either.  
It is therefore of great 
interest to be able to generalize in a different way 
the techniques developed for cellular 
resolutions, so that they become applicable to 
every minimal resolution. 
%
%
%
%
%

In this paper, we show that replacing CW-complexes with  posets 
yields such a generalization.  
We introduce the notion of a 
\emph{resolution supported on a poset}.
When the poset is a CW-poset, i.e. the face poset of a regular 
CW-complex, we recover the notion of cellular resolution. 
In particular, the examples by Reiner and Welker~\cite{ReWe} and 
Velasco~\cite{Ve} provide 
monomial ideals whose minimal free resolution \emph{cannot} be supported 
on a CW-poset. Nevertheless, 
in our first main result, Theorem~\ref{T:MFR-from-poset},  
we show that every minimal 
free resolution of a monomial ideal is in fact supported on some poset. 
In general, there are many possible choices for such a poset, 
and one cannot expect 
a good relationship between the combinatorics/topology of the poset 
and the properties of the ideal. However, 
the situation improves when one considers posets 
that are sufficiently close to being CW.   
A poset $P$ is a 
\emph{homology CW-poset} (or \emph{hcw-poset} for short) over 
a commutative ring with unit $\Bbbk$  if for each $a\in P$ the lower set   
$P_{<a}=\{x\in P| x<a\}$ is a finite-dimensional 
homology sphere over $\Bbbk$,   
i.e. the order simplicial complex $\Delta_{<a}=\Delta(P_{<a})$ has 
finite dimension and its reduced homology with coefficients in $\Bbbk$ 
is the same as that of a $(\dim\Delta_{<a})$-sphere.  
In our second main result, Theorem~\ref{T:cl-supports-I}, 
we prove that 
every monomial ideal in a polynomial ring over a field $\Bbbk$ 
has a minimal free resolution  
supported on an appropriate hcw-poset over $\Bbbk$. Again, in general 
such an hcw-poset is not unique, and it is an 
important open question whether 
one can be constructed in a canonical way. 
Our proof uses a series of non-canonical choices, and we identify 
in Theorem~\ref{T:BettihcwiffRigid} 
the class of monomial ideals for which no choices arise.  
This turns out to be the  
class of \emph{rigid} monomial ideals \cite{ClMa}. In this case 
the Betti poset \cite{ClMa,ClMa2,TchVa}
is an hcw-poset that supports a minimal free resolution,  
and this property characterizes rigid ideals. 

In addition to providing a new method for constructing explicit 
minimal free resolutions of monomial ideals, a key benefit of   
our results is that, by considering hcw-posets, they 
allow to translate and generalize, almost verbatim,   
cellular resolution techniques into methods that can be used on 
any minimal resolution. We include 
two applications that demonstrate this 
process. First, we leverage the methods from \cite{BaSt} and prove  
in our third main result, Theorem~\ref{T:toric-resolutons-are-hcw}, 
that minimal free resolutions of 
toric rings are supported on what we call \emph{toric} hcw-posets. 
Finally, in our second application 
we generalize notions and techniques from 
\cite{MiSt}, and give a new combinatorial proof of a fundamental
result of Miller~\cite{Mi}  on 
the structural relationship
between Artinianizations of monomial ideals and Alexander duality.

The outline of  the paper is as follows. 
In Section~\ref{S:Poset} we introduce  
the conic chain complex in Definition~\ref{D:conic-cx-of-poset}. 
This is a central 
object of study, and we provide two additional equivalent 
descriptions in Propositions~\ref{P:conic-form} 
and \ref{P:conic-elements}. Readers 
with background in pure combinatorics may find the second of these 
more appealing.     
The notion of a resolution supported 
on a poset is introduced in Section~\ref{S:Homogenization}. 
In Section~\ref{S:MinSupp}
we define bases with minimal support and describe their basic  
properties. 
In Section~\ref{S:Incidence}, 
we introduce and study incidence posets.  
In Section~\ref{S:hcw} we provide a mechanism for transforming 
an incidence poset into an hcw-poset. Finally, 
we put these results together in  
Section~\ref{S:Main}. 
The applications of the main results to 
Artinianizations and Alexander duality are discussed 
in Section~\ref{S:AAD}. 


Throughout this paper $\Bbbk$ is a commutative ring with unit, 
rings are commutative $\Bbbk$-algebras with unit, modules are 
unitary, and unadorned tensor products are over~$\Bbbk$. 
We  consider $\mathbb Z^m$ as a poset via the coordinatewise 
partial ordering: $(a_1,\dots,a_m)\le (b_1,\dots,b_m)$ if and only if 
$a_i\le b_i$ for all $i$. 
For a chain complex $\mathcal F_\bullet=(F_n, f_n)$ we denote by 
$\mathcal F_\bullet[m]=(F_n[m], f_n[m])$ its shift by $m$, where 
$F_n[m]=F_{n+m}$ and $f_n[m]=f_{n+m}$. The complex $\mathcal F_\bullet$ 
is called \emph{nonnegative} if $F_n=0$ for $n<0$. 
We write $Z_n\mathcal F_\bullet=\Ker f_n$ for the $n$-cycles,  
$B_n\mathcal F_\bullet= \im f_{n+1}$ for the $n$-boundaries, 
and $\HH_n\mathcal F_\bullet=Z_n\mathcal F_\bullet\big/ B_n\mathcal F_\bullet$ 
for the $n$th homology of $\mathcal F_\bullet$. 
We say that $\mathcal F_\bullet$ is 
\emph{acyclic} or a \emph{resolution} if $\HH_n\mathcal F_\bullet =0$ 
for $n\ne 0$, and call $\mathcal F_\bullet$  \emph{exact} 
if $\HH_n\mathcal F_\bullet=0$ for all $n$. 

\begin{acknowledgments}
We are grateful to the anonymous referee, whose comments 
helped significantly in improving the 
scope and clarity of our presentation. We also thank Ezra Miller
for pointing out that Theorem~\ref{T:main-theorem-3}
is already known as a consequence of~\cite[Theorem 4.5]{Mi}.  
\end{acknowledgments}

\section{The conic chain complex}\mlabel{S:Poset}

Let $P$ be a poset. We write $\hat P$ and $\check P$ 
for the posets obtained from 
$P$ by adding a maximum $\hat 1$ and a minimum $\check 0$, respectively. 
In particular, we have $(\check P)^{op}=\widehat{P^{op}}$. 
The \emph{order simplicial complex} $\Delta=\Delta(P)$ is the abstract 
simplicial complex with vertices the elements of $P$, and $n$-faces the 
chains
\begin{equation}\label{E:poset-chains}
a_0 > \dots > a_n 
\end{equation}
in $P$. We will always orient the faces of $\Delta$ 
by ordering their vertices in decreasing order 
(as defined in $P$), and we write $[a_0,\dots,a_n]$ for the 
oriented $n$-face given by the chain \eqref{E:poset-chains}. In 
particular, we write $[\,]$ for the oriented empty face of $\Delta$.  
The free $\Bbbk$-module with basis all (oriented) $n$-faces of $\Delta$ 
is denoted by $C_n=C_n(\Delta,\Bbbk)$ and called the module of (oriented   
simplicial)  $n$-chains of $\Delta$. For each $n\ge 0$ the $n$th 
\emph{simplicial boundary operator}  
$d_n\: C_n \lra C_{n-1}$ is given 
on basis elements by the formula \ 
$
d_n[a_0,\dots, a_n] = 
\sum_{i=0}^n (-1)^i[a_0,\dots,\hat a_i, \dots, a_n], 
$ 
\ 
where the hat means that the corresponding element is omitted. 
This yields the \emph{simplicial chain complex} of $\Delta$ (with 
coefficients in $\Bbbk$) 
\[
C_\bullet(\Delta,\Bbbk)\quad = \quad  
0 \lla C_0 \lla\dots\lla C_{n-1} \overset{d_n}\lla C_n \lla \dots \ , 
\]
and its augmentation 
\[
\widetilde C_\bullet(\Delta,\Bbbk) \quad = \quad 
0 \lla C_{-1} \overset{d_0}\lla  
C_0 \lla\dots\lla C_{n-1} \overset{\delta_n}\lla C_n \lla \dots \ . 
\]
Their $n$th homology modules are denoted by $\HH_n(\Delta,\Bbbk)$ 
and $\RH_n(\Delta,\Bbbk)$, respectively. 
If $L$ is a subposet of $P$ we have a canonical inclusion \  
$
\widetilde C_\bullet(\Delta(L),\Bbbk)\subseteq
\widetilde C_\bullet(\Delta(P),\Bbbk), 
$ 
the quotient \ 
$
C_\bullet(\Delta(P),\Delta(L),\Bbbk)= 
\widetilde C_\bullet(\Delta(P),\Bbbk)\big/
\widetilde C_\bullet(\Delta(L),\Bbbk)  
$
\ is called the \emph{simplicial chain complex of the pair} 
$\bigl(\Delta(P),\Delta(L)\bigr)$, and the $n$th 
\emph{relative homology} 
of this pair 
(with coefficients in $\Bbbk$) is  \ 
$
\HH_n(\Delta(P),\Delta(L), \Bbbk) = 
\HH_nC_\bullet(\Delta(P),\Delta(L),\Bbbk).  
$ 

Each element $a\in P$ generates a lower set  
$P_{\le a}=\{x\in P\mid x\le a\}$ and an \emph{open} lower set   
$P_{<a}=\{x\in P\mid x<a\}$, and  
we consider the 
order complexes 
$
\Delta_{\le a}=
\Delta(P_{\le a})
\text{ and }
\Delta_{<a}=
\Delta(P_{<a}). 
$
We call the number 
\[
d_P(a)=d(a)=\dim(\Delta_{\le a})
\] 
the \emph{dimension} of the element $a$ of $P$.  
For each $n$ we set 
\[
P^n=\{a\in P\mid d(a)\le n\},  
\]  
and we filter the order complex 
$\Delta$ by the subcomplexes 
\[
\Delta^n= 
\Delta(P^n)=
\bigcup_{d(a)\le n}\Delta_{\le a}.  
\]
As each $\Delta_{\le a}$ with $d(a)=n$ is a cone with apex $a$ 
over the base $\Delta_{<a}$, and $\Delta^n$ is obtained by attaching 
each such $\Delta_{\le a}$ to $\Delta^{n-1}$ along the inclusion 
$\Delta_{<a}\subseteq \Delta^{n-1}$, it seems appropriate that we 
refer to $\Delta^n$ as the \emph{conic $n$-skeleton} 
of $\Delta$. This induces 
a canonical bounded below filtration on the simplicial chain complex 
$C_\bullet(\Delta, \Bbbk)$, 
and in turn a canonical (fourth quadrant in this particular case)
spectral sequence, see e.g. \cite[Section 5.4]{Wei},   
which has in first page as horizontal strands the complexes 
\[
\mathcal C_{q,\bullet}(P,\Bbbk)\ = \quad 
0\lla \HH_0(\Delta^{q},\Delta^{q-1},\Bbbk)\lla \dots \lla 
      \HH_n(\Delta^{n+q},\Delta^{n+q-1},\Bbbk)\lla \dots 
\]
at the $y=-q$ level. 
Of particular importance for us will 
be the strand at $q=0$.

\begin{definition}\mlabel{D:conic-cx-of-poset}
We set 
$
\mathcal C_\bullet(P, \Bbbk)=
\mathcal C_{0,\bullet}(P, \Bbbk),  
$
and we call $\mathcal C_\bullet(P, \Bbbk)$ the 
\emph{conic chain complex} of $P$ over $\Bbbk$. 
\end{definition}

\begin{example}\label{Ex:basic-conic-example}  
Consider the poset structure $P$ on the set $\{a,b,c,\hat 1\}$ with 
Hasse diagram displayed below. 

     \begin{center}
     \begin{tikzpicture}[scale=0.9, vertices/.style={draw, fill=black, circle, inner sep=1pt}]
             \node [vertices, label=left:{$\smash{a}$}] (1) at (-1.5,1){};
             \node [vertices, label=right:{$\smash{b}$}] (2) at (0,1){};
             \node [vertices, label=right:{$\smash{c}$}] (3) at (1.5,1){};
             \node [vertices, label=above:{$\hat{1}$}] (123) at (0,2){};
     \foreach \to/\from in {1/123, 2/123, 3/123}
     \draw [-] (\to)--(\from);
     \end{tikzpicture}
     \end{center}

\medskip
\noindent
Here $P^{-1}=\emptyset$ and $P^0=\{a,b,c\}$, while $P^n=P$ for $n\ge 1$. 
Thus $\HH_0(\Delta^0,\Delta^{-1},\Bbbk)$ is 
naturally identified with the free submodule of    
$C_0(\Delta, \Bbbk)$ with basis $\{[a], [b], [c]\}$, 
and $\HH_1(\Delta^1,\Delta^0,\Bbbk)$ is 
naturally identified with the free submodule of 
$C_1(\Delta, \Bbbk)$ with basis 
$\{ [\hat 1, a] - [\hat 1, b], \ [\hat 1,b] - [\hat 1,c]\}$. 
Now the description of the maps in the first page 
of the spectral sequence of a filtration, see 
\cite[The Construction 5.4.6]{Wei}, shows that 
the differential $\partial_1$ of the conic chain complex 
$\mathcal C_\bullet(P,\Bbbk)$ is obtained 
simply by restriction from the differential of $C_\bullet(\Delta, \Bbbk)$. 
Thus our conic chain complex can be written, using the above bases, as 
\begin{equation*}
   \begin{CD} 
    {0 \longleftarrow \Bbbk^3} 
	@< {
            \begin{bsmallmatrix*}[r]
              1 &  0 \\
             -1 &  1 \\
              0 & -1    
            \end{bsmallmatrix*}
	} <\partial_1<
    {\Bbbk^{2}} \longleftarrow 0.  
   \end{CD}
\end{equation*}
Note also that $P$ is not an hcw-poset since 
$
\RH_0(\Delta_{<\hat 1}, \Bbbk)\cong 
\HH_1(\Delta^1, \Delta^0,\Bbbk)\cong \Bbbk^2$.      
\end{example}

In order to study the properties of  the conic chain complex, 
we need the following proposition, which is an immediate
consequence of the description, 
see \cite[The Construction 5.4.6]{Wei},    
of the maps in the first page 
of the spectral sequence of a filtration.

\begin{proposition}\label{P:conic-form} 
The conic chain complex 
$\mathcal C_\bullet(P, \Bbbk)$ has the form  
\[
0\la \HH_0(\Delta^{0},\Delta^{-1},\Bbbk)
 \la \dots \la \HH_{n-1}(\Delta^{n-1},\Delta^{n-2},\Bbbk)
 \overset{\partial_n}\lla 
      \HH_n(\Delta^{n},\Delta^{n-1},\Bbbk)\la \dots 
\]
and for $n\ge 1$ the map $\partial_n$ 
is the composition $\partial_n=\iota_{n-1}\delta_n$ 
where $\iota_n$ and $\delta_n$ are the 
canonical maps 
\[
\RH_n(\Delta^n,\Bbbk)\overset{\iota_n}\lra 
\HH_n(\Delta^n,\Delta^{n-1},\Bbbk)\overset{\delta_n}\lra 
\RH_{n-1}(\Delta^{n-1},\Bbbk)
\]
in the long exact sequence for relative homology 
of the pair $(\Delta^n,\Delta^{n-1})$.
\end{proposition}

In particular, since $\iota_n$ is injective for each $n$,  
we use it to identify $\RH_n(\Delta^n,\Bbbk)$ as 
a $\Bbbk$-submodule of $\HH_n(\Delta^n,\Delta^{n-1},\Bbbk)$, 
and thus we obtain for $n\ge 1$ the important equality 
\begin{equation}\elabel{E:conic-kernel}
\Ker\partial_n = \RH_n(\Delta^n, \Bbbk). 
\end{equation} 
Furthermore, for any $s$, whenever $a\ne b$ with $d(a)=d(b)=s$, 
we have an inclusion 
$\Delta_{\le a}\cap\Delta_{\le b}\subseteq \Delta^{s-1}$  and 
therefore  
\begin{equation}\elabel{E:rel-homology}
\HH_n(\Delta^{n+q},\Delta^{n+q-1},\Bbbk)\ = \negthickspace 
{\bigoplus_{d(a)=n+q}}
\negthickspace
\HH_n(\Delta_{\le a},\Delta_{<a},\Bbbk) \ \cong \negthickspace  
{\bigoplus_{d(a)=n+q}}
\negthickspace
\RH_{n-1}(\Delta_{<a},\Bbbk). 
\end{equation}

\begin{remark}\label{R:conic-not-cellular}  
As Proposition~\ref{P:conic-form} shows, 
the conic chain complex of a poset resembles in many 
ways the cellular chain complex of a CW-complex. 
Recall that if $X$ is a CW-complex with $i$-skeleton $X^i$ then the 
set $A_i$ of homology classes $\langle\sigma\rangle$ of the 
characteristic maps for the 
$i$-cells $\sigma$ of $X$ (with closure $\bar\sigma$ and boundary 
$\dot\sigma$) form a basis of 
$
\HH_i(X^i,X^{i-1},\Bbbk)\cong
\bigoplus_{\langle\sigma\rangle\in A_i}\HH_i(\bar\sigma,\dot\sigma, \Bbbk).  
$
For an $i$-cell $\sigma$ we then have a unique expression 
$
d_i(\langle\sigma\rangle)=
\sum_{\langle\tau\rangle\in A_{i-1}}a_{\sigma,\tau}\langle\tau\rangle
$
where $d_i$ is the differential in the cellular chain complex of $X$ 
over $\Bbbk$. The coefficient $a_{\sigma,\tau}\in\Bbbk$ is called the 
\emph{incidence coefficient} of $\sigma$ and $\tau$. We set $\sigma>\tau$ 
exactly when $a_{\sigma,\tau}\ne 0$ in $\Bbbk$. This generates a poset 
structure $P(X)$ on the set  
of all open cells of $X$ that we will refer 
to in the sequel as the \emph{face poset} of $X$ (over $\Bbbk$). 
We will see in 
Example~\ref{E:cellular-is-conic} that in the special case 
of a CW-poset $P$ its 
conic chain complex is indeed the cellular chain complex of a regular 
CW-complex that has $P$ as face poset
(recall that a CW-complex is \emph{regular} if the characteristic 
maps for its cells are homeomorphisms). However, in general there 
seems to be no natural way to 
obtain one of these two chain complex constructions from the other. 
For instance, from our Theorem~\ref{T:cl-supports-I} and  
Velasco's results \cite{Ve} the interested reader can deduce the 
existence of an hcw-poset $P$ whose conic chain complex 
\emph{cannot} be obtained 
as the cellular chain complex of any CW-complex having 
$P$ as face poset. 
Conversely, the face poset of a non-regular CW-complex in general no
longer reflects the underlying topology, and in particular its cellular 
chain complex cannot always be recovered from the poset alone.    
\end{remark}

Next, we describe in more detail 
the connection between the conic chain complex 
$\mathcal C_\bullet(P, \Bbbk)$ and 
the simplicial chain complex 
$C_\bullet\bigl(\Delta(P), \Bbbk\bigr)$. 
Notice that, just as in Example~\ref{Ex:basic-conic-example}, 
by dimensional considerations  
the submodule $Z_n(\Delta^n, \Delta^{n-1}, \Bbbk)$ of 
relative $n$-cycles inside $C_n(\Delta^n, \Bbbk)$ 
is isomorphic 
to $\HH_n(\Delta^n,\Delta^{n-1},\Bbbk)$. Combined with the 
canonical inclusion of oriented simplicial $n$-chains 
$C_n(\Delta^n, \Bbbk) \lra C_n(\Delta, \Bbbk)$, this induces 
for each $n$ a canonical inclusion map of $\Bbbk$-modules 
\begin{equation}\label{E:conic-inclusion}
\HH_n(\Delta^n,\Delta^{n-1},\Bbbk) \lra 
C_n(\Delta, \Bbbk).  
\end{equation}
The following proposition is now immediate from the 
definitions and Proposition~\ref{P:conic-form}. 

\begin{proposition}\label{P:conic-inclusion}
The maps \eqref{E:conic-inclusion} identify canonically 
the conic chain complex 
$\mathcal C_\bullet(P, \Bbbk)$ as a subcomplex of the simplicial 
chain complex  $C_\bullet(\Delta, \Bbbk)$. 
\end{proposition}

We will need an even more explicit description of the 
elements of $\mathcal C_\bullet(P, \Bbbk)$ inside 
$C_\bullet(\Delta, \Bbbk)$, and for this we introduce 
some notation. Recall that we orient the faces of 
$\Delta$ by 
ordering their vertices in decreasing order, and  
we write $[\thinspace]$ for the oriented empty face 
of $\Delta$. 
If $\sigma=[a_0,\dots, a_n]$ is an oriented face of 
$\Delta_{<a}$, then we write $[a,\sigma]$ for the oriented 
face $[a, a_0, \dots, a_n]$ of $\Delta_{\le a}$; in 
particular we have 
$
\bigl[ a, [\thickspace]\bigr] = 
[a].
$ 
Also, if 
$
w=\sum c_\sigma\sigma
$  
is an $n$-chain of $\Delta_{<a}$ then we write 
$[a, w]$ for the $(n+1)$-chain of $\Delta_{\le a}$ 
\begin{equation}\elabel{E:cone-over-a}
[a, w] = \sum c_\sigma [a, \sigma]. 
\end{equation}
With this notation the next proposition 
is straightforward to verify. 

\begin{proposition}\label{P:conic-elements} 
The elements of the conic chain complex 
$\mathcal C_\bullet(P,\Bbbk)$ that 
belong to a component 
$\HH_n(\Delta_{\le a}, \Delta_{<a}, \Bbbk)$ 
with $d(a)=n$ are exactly all
elements of $C_n(\Delta, \Bbbk)$ of the form 
$[a, z]$ where $z$ is an $(n-1)$-cycle of 
$\Delta_{<a}$. 
\end{proposition}

In particular, we have 
\[
\mathcal C_0(P,\Bbbk)=
\bigoplus_{d(a)=0}\Bbbk\cdot[a], 
\]
and therefore the simplicial boundary operator 
$
C_0(\Delta, \Bbbk) \lra 
C_{-1}(\Delta, \Bbbk) = \Bbbk\cdot[\thinspace] 
$
induces by restriction an augmentation map 
\[
\partial_0\: 
\mathcal C_0(P, \Bbbk) \lra 
\Bbbk\cdot[\thinspace].  
\]

\begin{definition}
We set 
$\mathcal C_{-1}(P, \Bbbk)=\Bbbk\cdot[\thinspace]$ 
and we write $\widetilde{\mathcal C}_\bullet(P,\Bbbk)$ 
for the chain complex obtained by augmenting 
$\mathcal C_\bullet(P, \Bbbk)$ with the map 
$
\partial_0\: 
\mathcal C_0(P, \Bbbk)\lra \mathcal C_{-1}(P, \Bbbk).
$  
\end{definition}

\begin{remarks}\label{R:augmented-conic}
(a) 
It is straightforward to check that 
\eqref{E:conic-kernel} holds also for $n=0$. 

(b) 
An action of a group $G$ on $P$ induces in a canonical way an 
action of $G$ on the chain complexes 
$\mathcal C_\bullet(P,\Bbbk)$ and 
$\widetilde{\mathcal C}_\bullet(P,\Bbbk)$. 

(c) 
If $L$ is a lower set in $P$, then 
$\widetilde{\mathcal C}_\bullet(L,\Bbbk)$ is naturally a subcomplex 
of $\widetilde{\mathcal C}_\bullet(P,\Bbbk)$. Indeed, this is clear from 
Proposition~\ref{P:conic-elements} since for each $a\in L$ we 
have $\Delta(L_{\le a})=\Delta(P_{\le a})$, hence also 
$\Delta(L_{< a})=\Delta(P_{< a})$ and $d_L(a)=d_P(a)$.   
\end{remarks}

Next, we identify useful conditions under which the homology of 
the conic chain complex 
$\mathcal C_\bullet(P,\Bbbk)$ reflects the topology of $P$.     

\begin{proposition}\mlabel{P:conic-simplicial} 
Suppose that for all $a\in P$ we have that $d(a)$ is finite, and  
$\RH_m\bigl(\Delta_{<a}, \Bbbk\bigr)=0$ for $m\le d(a)-2$. 

Then the canonical inclusion \ 
$\mathcal C_\bullet(P,\Bbbk)\lra C_\bullet\bigl(\Delta(P), \Bbbk\bigr)$ \  
from Proposition~\ref{P:conic-inclusion} is a homology isomorphism. 
\end{proposition}

\begin{proof}
Under the assumptions on $P$ our bounded below filtration of  the 
simplicial chain complex $C_\bullet(\Delta(P), \Bbbk)$ is exhaustive, hence   
by \cite[Classical Convergence Theorem 5.5.1]{Wei} our spectral sequence 
converges to the homology of $\Delta(P)$. Furthermore,  
the assumptions on $P$ 
imply 
$\HH_n(\Delta_{\le a}, \Delta_{<a}, \Bbbk)= 0$ for $n\ne d(a)$.
Thus, in view of \eqref{E:rel-homology}, our spectral 
sequence collapses on the first page to the conic chain 
complex $\mathcal C_\bullet(P, \Bbbk)$. Next, 
by Proposition~\ref{P:conic-elements} the filtration induced 
on $\mathcal C_\bullet(P,\Bbbk)$ via the inclusion map is just the 
standard filtration along homological degree. Therefore 
the corresponding spectral sequence also 
collapses on the first page to the conic chain complex, and the 
map of spectral sequences induced by the inclusion of chain complexes 
is an isomorphism from the first page on.   
\end{proof}

\begin{example}\label{E:NonCWPoset}
Taking the poset $P$ whose Hasse diagram is 
     \begin{center}
     \begin{tikzpicture}[scale=0.9, vertices/.style={draw, fill=black, circle, inner sep=1pt}]
             \node [vertices, label=right:{}] (1) at (-2.25+0,1.33333){};
             \node [vertices, label=right:{}] (2) at (-2.25+1.5,1.33333){};
             \node [vertices, label=right:{}] (3) at (-2.25+3,1.33333){};
             \node [vertices, label=right:{}] (4) at (-2.25+4.5,1.33333){};
             \node [vertices, label=right:{}] (12) at (-3+0,2.66667){};
             \node [vertices, label=right:{}] (13) at (-3+1.5,2.66667){};
             \node [vertices, label=right:{}] (23) at (-3+3,2.66667){};
             \node [vertices, label=right:{}] (24) at (-3+4.5,2.66667){};
             \node [vertices, label=right:{}] (34) at (-3+6,2.66667){};
             \node [vertices, label=right:{}] (1234) at (-0+0,4){};
     \foreach \to/\from in {1/12, 1/13, 2/12, 2/23, 2/24, 3/13, 3/23, 3/34, 4/24, 4/34, 12/1234, 13/1234, 23/1234, 24/1234, 34/1234}
     \draw [-] (\to)--(\from);
     \end{tikzpicture}
     \end{center}
we see that every open lower set $P_{<a}$ is a homology 
sphere, except for the open lower set $P_{<\hat 1}$ generated 
by the maximal element $\hat 1$, where we have 
$\RH_1(\Delta_{<\hat 1}, \Bbbk)\cong \Bbbk^2$ and 
$\RH_i(\Delta_{<\hat 1}, \Bbbk)=0$ for $i\ne 1$.  
Thus $P$ satisfies the assumptions of 
Proposition~\ref{P:conic-simplicial} but is not an 
hcw-poset. 
Under one choice of bases for the corresponding relative homology modules, 
the conic chain complex of $P$ takes the form  
\begin{equation*}
   \begin{CD} 
    {0 \leftarrow \Bbbk^4} 
	@< {
            \begin{bsmallmatrix*}[r]
             -1 & 0 &  -1 & -1 & 0 \\
              0 &  -1 & 1 &  0 &  -1 \\
              1  & 1 &  0  & 0 &  0   \\
              0 &  0 &  0 &  1  & 1   \\
            \end{bsmallmatrix*}
	} <\partial_1<
    {\Bbbk^{5}} 
	@< {
               \begin{bsmallmatrix*}[r]
               -1 & 0  \\
               1 & 0 \\
               1 & 1 \\
               0 & -1  \\
               0 & 1 \\
            \end{bsmallmatrix*}
	} <\partial_2< 
    {\Bbbk^{2}} \leftarrow 0,  
   \end{CD}
\end{equation*}
and it is straightforward to check that it is acyclic. 
Of course, this also follows from Proposition~\ref{P:conic-simplicial} 
as $\Delta(P)$ is a cone with apex $\hat 1$. 
\end{example}

\begin{remarks}\mlabel{T:conic-poset}
(a) 
Here, and in Examples~\ref{E:P-grading-examples}(b,c) and~\ref{E:cellular-is-conic}, and 
Remarks~\ref{R:conic-poset-resolution}, 
we discuss the relationship between the 
conic chain complex and the \emph{poset construction} 
studied in \cite{Cl}. These observations are only used 
to compute examples and 
to point out that the notion of cellular resolution is just 
a special case of our notion of resolution supported on a 
poset, and play no role in the proofs of our main results. 
Thus, in order to streamline our presentation, rather than 
recalling them here, we refer the reader to \cite{Cl} 
and \cite{ClTch} for 
the definition and basic properties of the poset 
construction. 

(b)  
Suppose $P$ is a finite poset.   
It is immediate from the description of the maps in the 
first page of the spectral sequence of a filtration, 
see \cite[The Construction 5.4.6]{Wei}, and the 
definitions in \cite{Cl}, that 
the chain complex 
\smash[b]{
$   
\bigoplus_{q\ge 0} \mathcal C_{q,\bullet}(P, \Bbbk)
$
}
is canonically isomorphic to the part in non-negative 
homological degrees of the shifted poset construction 
\smash[t]{$\mathcal D_\bullet(\widetilde P,\Bbbk)[1]$} 
from \cite{Cl}, where \smash[t]{$\widetilde P$} is the
ranked poset obtained from $P$ 
by using the canonical procedure from 
\cite[Proposition~A.9]{Cl}.  Thus, 
when the poset $P$ satisfies the assumptions of 
Proposition~\ref{P:conic-simplicial} we have a 
canonical isomoprhism 
\[
\widetilde{\mathcal C}_\bullet(P, \Bbbk) \cong
\mathcal D_\bullet(\widetilde P,\Bbbk)[1].
\]
In particular, this holds when $P$ is a homology CW-poset.  
\end{remarks}

\section{Homogenization, and resolutions supported 
on posets}\mlabel{S:Homogenization}

Let $P$ be a poset structure on a set $B$. Consider a chain complex of 
$\Bbbk$-modules 
\[
\mathcal F_\bullet \ = \qquad 
                 \dots \lla F_{n-1}
          \overset{f_n}\lla F_n \lla \dots 
\] 
A \emph{$P$-grading} on $\mathcal F_\bullet$ 
is a direct sum decomposition of $\Bbbk$-modules 
$
F_n = \bigoplus_{a\in B} F_{n,a}
$
for each $n$, such that for all $a,b\in B$ and all 
$x\in F_{n,a}$ the component $f_n^b(x)$ of $f_n(x)$ 
inside $F_{n-1,b}$ is zero when $b\not\leq a$ in $P$.

\begin{examples}\mlabel{E:P-grading-examples}
(a) 
For any poset $P$ 
the conic chain complex $\mathcal C_\bullet(P,\Bbbk)$ 
is naturally $P$-graded with the component 
$F_{n,a}$ zero when $n\ne d(a)$ and with 
$F_{d(a),a}=\HH_{d(a)}(\Delta_{\le a},\Delta_{<a},\Bbbk)$.  
This extends naturally to a $\smash[t]{\check P}$-grading 
of $\smash[t]{\widetilde{\mathcal C}_\bullet(P,\Bbbk)}$ by setting 
$F_{n,\check 0}=0$ for $n\ne -1$ and 
$F_{-1,\check 0}={\mathcal C}_{-1}(P,\Bbbk)$. 

(b) 
For any finite poset $P$ the shifted poset construction 
complex $\smash[t]{\mathcal D_\bullet(\widetilde P, \Bbbk)[1]}$, see  
Remark~\ref{T:conic-poset}, is naturally $\smash[t]{\check P}$-graded with  
$
\smash[t]{F_{n,a}= \RH_{n-1}(\Delta_{<a},\Bbbk)}   
$
for $a\ne \check 0$, and with $F_{-1,\check 0}=F_{-1}$ and $F_{n,\check 0}=0$ 
for $n\ne -1$.  
When $P$ is a homology CW-poset the canonical 
isomorphism from Remark~\ref{T:conic-poset}(b) respects 
the $\smash[t]{\check P}$-grading.

(c) 
The cellular chain complex $C_\bullet(X,\Bbbk)$ of 
a CW-complex $X$ with coefficients in $\Bbbk$:  
here $F_n=\HH_n(X^n, X^{n-1},\Bbbk)$, the poset 
$P=P(X)$ is the face poset of $X$, see Remark~\ref{R:conic-not-cellular}, 
and for an (open) cell 
$\sigma\in P$ with closure $\bar\sigma$ and boundary $\dot\sigma$ 
the component $F_{n,\sigma}$ is zero if 
$\dim\sigma\ne n$ and 
$
F_{\dim\sigma, \sigma}=
\HH_{\dim\sigma}(\bar\sigma,\dot\sigma, \Bbbk). 
$ 
This extends naturally to a $\smash[t]{\check P}$-grading on 
$\smash[t]{\widetilde C_\bullet(X,\Bbbk)}$ by setting $F_{n,\check 0}=0$ 
for $n\ne -1$, and $F_{-1,\check 0}=C_{-1}(X,\Bbbk)$. 
In the case when $X$ is finite and regular we also have 
$\smash[t]{P=\widetilde P}$, and by construction 
the canonical isomorphism 
\[
\widetilde C_\bullet(X, \Bbbk) \lra \mathcal D_\bullet(P, \Bbbk)[1]
\]
from \cite[Theorem~1.7]{ClTch} respects the $\check P$-grading.

(d) 
For any $P$ the simplicial chain complex 
$C_\bullet(\Delta, \Bbbk)$ is naturally $P$-graded,  
where $F_{n,a}$ is the subspace of $C_n(\Delta, \Bbbk)$ 
with basis all oriented $n$-faces of $\Delta$ of the 
form $[a, a_1, \dots, a_n]$. 
This extends naturally to a $\smash[t]{\check P}$-grading on 
$\smash[t]{\widetilde C_\bullet(\Delta,\Bbbk)}$ by setting $F_{n,\check 0}=0$ 
for $n\ne -1$, and $F_{-1,\check 0}=C_{-1}(\Delta,\Bbbk)$. 
In particular, the canonical 
inclusion 
\[
\widetilde{\mathcal C}_\bullet(P, \Bbbk) \lra 
\widetilde C_\bullet(\Delta, \Bbbk)
\]
respects the $\check P$-grading. 
%

(e) 
The $T$-complex $T_\bullet(\phi)$ from \cite{Tch} 
associated with a representation 
$\phi$ of a matroid $\bf M$ over a field $\Bbbk$: here 
the poset $P$ is the lattice of T-flats of $\bf M$, 
and for a T-flat $A$ of $\bf M$ the component 
$F_{n,A}$ is zero when the level of $A$ is not $n$, 
and when the level is $n$ we have $F_{n,A}=T_A(\phi)$, 
the T-space of the T-flat $A$ with respect to the 
matroid representation map $\phi$. 
\end{examples}

\begin{definition}
A morphism of posets 
$
\deg\: P \lra \mathbb Z^m 
$
is called a \emph{$\mathbb Z^m$-labeling}, or simply a 
\emph{labeling} of $P$. We call the pair $Z=(P,\deg)$ a 
\emph{labeled poset}, and write 
$a\in Z$ when $a$ is an element of $P$. 
For an element $\alpha\in\mathbb Z^m$ we denote by  
$P_{\deg\le \alpha}$ 
the set of elements $x\in Z$ such that $\deg(x)\le \alpha$, 
and by $Z_{\le\alpha}$ the labeled poset $(P_{\deg\le\alpha}, \deg)$.  
\end{definition} 

Let $Z=(P,\deg)$ be a labeled poset. 
We review the formalism of homogenizing a $P$-graded 
chain complex $\mathcal F_\bullet$ with respect to the labeled poset $Z$.  
This is a standard technique for studying 
resolutions of monomial ideals and multigraded modules, 
see e.g. \cite{BaPeSt, BaSt, BaWe, PeVe} for cases when 
each component $F_{n,a}$ is free of rank $\le 1$, and 
\cite{ChTch, Tch, Cl} for other cases. 

Let $R=\Bbbk[x_1,\dots, x_m]$ be a polynomial ring over 
$\Bbbk$ with the standard $\mathbb Z^m$-grading. 
If $M=\bigoplus_{\alpha\in\mathbb Z^m}M_\alpha$ is a 
$\mathbb Z^m$-graded $R$-module and $\gamma\in\mathbb Z^m$ 
then the \emph{twist} $M(\gamma)$ 
is the $\mathbb Z^m$-graded $R$-module with 
$M(\gamma)_\alpha=M_{\gamma+\alpha}$.   
The \emph{homogenization} of $\mathcal F_\bullet$ (with 
respect to the labeled poset $Z$) is the chain complex 
$\widehat{\mathcal F}_\bullet$ of $\mathbb Z^m$-graded 
$R$-modules 
\[
\widehat{\mathcal F}_\bullet \ = \qquad 
                      \dots \lla \widehat F_{n-1}
          \overset{\hat f_n}\lla \widehat F_n \lla \dots \ , 
\] 
where for each $n$ we have 
$
\widehat F_n=\bigoplus_{a\in Z} F_{n,a}\otimes_\Bbbk R(-\deg a),  
$
and for $x\in F_{n,a}$ one has 
\[
\hat f_n(x\otimes 1) = 
\sum_{b\le a} f_n^b(x)\otimes x^{\deg a -\deg b}.  
\]
In particular, for $\alpha\in\mathbb Z^m$ the homogeneous 
component of $\widehat F_n$ of degree $\alpha$ is the 
$\Bbbk$-submodule 
\begin{equation}\elabel{E:homogeneous-components}
\widehat F_{n,\alpha}= 
\bigoplus_{a\in Z_{\le\alpha}}F_{n,a}\otimes x^{\alpha-\deg a}. 
\end{equation}

We are now ready to introduce the main new notion of this 
paper.

\begin{definition}\mlabel{D:poset-support}
Let $Z=(P,\deg)$ be a labeled poset. 
We write $\mathcal C_\bullet(Z,\Bbbk)$ for the homogenization 
of $\mathcal C_\bullet(P,\Bbbk)$ with respect to $Z$ and call it 
the \emph{conic chain complex} of the labeled poset $Z$. 
We say that a chain complex $\mathcal F_\bullet$ of 
$\mathbb Z^m$-graded free $R$-modules is 
\emph{supported on} $Z$ if 
$\mathcal F_\bullet$ is isomorphic to 
the conic chain complex $\mathcal C_\bullet(Z, \Bbbk)$. 
We say that a (unlabeled) poset $P$ \emph{supports} $\mathcal F_\bullet$ if 
there exists a labeling $\deg$ of $P$ such that $\mathcal F_\bullet$ 
is supported on the labeled poset $(P,\deg)$. 
\end{definition}

\begin{example}\label{E:cellular-is-conic}
Let $X$ be a finite regular CW-complex with face poset $P$. Then  
a chain complex $\mathcal F_\bullet$ is supported on $X$ 
in the sense of \cite{BaSt} exactly if it is supported on 
the face poset $P$ in our sense; in particular the 
notion of cellular resolution is just a special case of 
the notion of resolution supported on a poset. Indeed, 
by \cite{Bjo} the poset $P$ is an hcw-poset over $\Bbbk$, hence 
Remark~\ref{T:conic-poset}(b) and 
Examples~\ref{E:P-grading-examples}(b) and 
\ref{E:P-grading-examples}(c) yield 
that the conic chain complex 
$\mathcal C_\bullet(P,\Bbbk)$ and the cellular chain complex 
$C_\bullet(X, \Bbbk)$ are isomorphic as $P$-graded chain 
complexes of $\Bbbk$-modules.   
\end{example}

\begin{remarks}\label{R:conic-poset-resolution}
There is a subtle distinction between the notion of 
\emph{poset resolution} discussed in \cite{Cl} and  
our notion of  complex supported 
on a poset: 

(a) 
For example, 
when $char(\Bbbk)=2$ the ideal 
{\small
\[
I =
( x_1x_2x_3,  x_1x_3x_5,  x_1x_4x_5, 
 x_2x_3x_4,  x_2x_4x_5, 
 x_1x_2x_6,  x_1x_4x_6, 
 x_2x_5x_6,  x_3x_4x_6,  x_3x_5x_6 )
\]
}
\negmedspace
\negthinspace
in $\Bbbk[x_1,\ldots,x_6]$ 
has  minimal free resolution supported in our sense on 
the poset $P(\mathcal F_\bullet,B)$ from Example~\ref{E:PPincidence}, 
but this poset 
does not produce a poset resolution.  
We examine this ideal in greater detail 
beginning in Example~\ref{E:PPRes}. 

(b)
On the other hand, the ideal 
$(axyz,bxyz,cxyz,dx,abcdy,abcdz)$ 
in the polynomial ring 
$\Bbbk[a,b,c,d,x,y,z]$ has minimal resolution which 
is not supported in our sense on its lcm-lattice \cite{GaPeWe}, but 
its lcm-lattice produces a 
poset resolution using the construction of \cite{Cl}. 

(c)
The two notions agree precisely when $P$ is a finite poset 
satisfying the conditions of Proposition~\ref{P:conic-simplicial}, 
in particular they agree when $P$ is a finite hcw-poset.   
\end{remarks}

The following standard fact will be needed later. 

\begin{proposition}\label{acyclicity-criterion}
Let $Z=(P,\deg)$ be a labeled poset. 
The conic chain complex $\mathcal C_{\bullet}(Z,\Bbbk)$ 
is acyclic
if and only if 
for each $\alpha\in\mathbb Z^m$ 
the chain complex 
$\mathcal C_\bullet(P_{\deg\le\alpha}, \Bbbk)$
is acyclic. 
\end{proposition}

\begin{proof}
From degree considerations it is clear that the homogenization 
$\widehat{\mathcal F}_\bullet$ of $\mathcal C_\bullet(P,\Bbbk)$ 
is a resolution precisely when 
each graded strand 
\[
\widehat{\mathcal F}_{\bullet,\alpha} \ = \quad 
0\lla \widehat F_{0,\alpha} \lla \dots \lla \widehat F_{n-1,\alpha}
          \overset{\hat f_n}\lla \widehat F_{n,\alpha} \lla \dots 
\] 
is a resolution. However, by 
\eqref{E:homogeneous-components} and the definitions, 
$\widehat{\mathcal F}_{\bullet,\alpha}$ can be canonically 
identified with $\mathcal C_\bullet(P_{\deg\le\alpha}, \Bbbk)$. 
\end{proof}

\section{Bases with minimal support}\mlabel{S:MinSupp}

Let $R$ be a 
ring,  
and consider a nonnegative chain complex  
\[
\mathcal F_\bullet \ = \quad 
0\lla F_0 \lla \dots \lla F_{n-1}
          \overset{f_n}\lla F_n \lla \dots 
\]
of free $R$-modules; we do not assume that 
$\mathcal F_\bullet$
is acyclic. 
Let 
\[
f_0 \: F_0 \lra M 
\] 
be 
an epimorphism of $R$-modules with $\im f_1\subseteq\Ker f_0$. 
{\bf If} it is given that $\mathcal F_\bullet$ is a resolution of $M$, 
{\bf then}, \emph{unless otherwise specified}, 
we will identify $M$ with $\coker(f_1)$ and 
take $f_0$ to be the canonical projection.     
For each $n\ge 0$ let $B_n$ 
be a basis of $F_n$, and let $B=\coprod B_n$ be their 
disjoint union. For any 
$y=\sum_{c\in B_n}a_c\, c\in F_n$ the \emph{support} of $y$ 
(with respect to $B$ or $B_n$) is the set
\[
\supp y = \supp_B y = \supp_{B_n}(y) = 
\{c\in B_n\mid a_c\ne 0\}. 
\]

\begin{definition}\mlabel{D:basis-minimal-support}
(a) When $y$ is in $F_n$  
we say that $y$ is a
\emph{cycle with minimal support} relative to the 
basis $B$  if $y$ is in $\Ker f_n$  
and the support of $y$ does not 
contain properly the support of any nonzero element 
of $\Ker f_n$.  

(b) If $y\in F_n$ with $n\ge 1$ the support of $f_n(y)$ 
in $F_{n-1}$ is called the \emph{boundary support} of $y$. 

(c) We say that $B$ is a 
\emph{basis with minimal (boundary) support} for $\mathcal F_\bullet$ 
if for each $n\ge 1$ and each $b\in B_n$ the element  
$f_n(b)\in F_{n-1}$ is a cycle with minimal support relative 
to the basis $B$. 
\end{definition}

\begin{remarks}\mlabel{R:change-basis}
(a) It is clear that if we replace each $b\in B$ by an 
associate $b'$, then the resulting bijective correspondence 
between $B$ and the new basis $B'$ commutes with taking 
supports and is an isomorphism of incidence posets (see 
Definition~\ref{D:incidence-poset}). In    
particular $B'$ is a basis with minimal support if and    
only if $B$ is.    

(b) 
Suppose $R$ is a field. Then  $y\in F_n$ is a cycle with 
minimal support if and only if the kernel $Z^y$ of 
the restriction of $f_n$ to 
$
\bigoplus_{b\in\supp y} Rb
$
is exactly $Ry$. Indeed, let $y=\sum_{b\in\supp y}a_b b$. 
If $Z^y=Ry$ then every non-zero cycle with support contained 
in $\supp y$ is a nonzero multiple of $y$ hence cannot have 
support properly contained in $\supp y$. Conversely, assume $y$ 
is a cycle of minimal support and let $0\ne z\in Z^y$. Then 
$z=\sum_{b\in\supp y}c_b b$ with $c_b\ne 0$ for some $b\in\supp y$, and  
therefore $a_b^{-1}c_b y - z$ has support properly contained in 
$\supp y$ hence equals $0$. 

(c) It is a standard exercise in linear algebra that when 
$R$ is a field, and $\mathcal F_\bullet$ is a resolution of $M$, 
then $\mathcal F_\bullet$ has a basis with minimal support. 

(d)
Suppose $R$ is a field, and $L$ is a proper $R$-submodule of 
$\Ker f_j$ for some $j\ge 0$. 
Then there exists in $\Ker f_j$ a cycle of minimal support that is 
not an element of $L$. Indeed let $y\in \Ker f_j\setminus L$ be an 
element with 
support minimal among the supports of elements of $\Ker f_j\setminus L$. 
If $z\ne 0$ is a cycle with support properly contained in $\supp y$, 
then we must have $z\in L$, hence for an appropriate $0\ne c\in R$ 
the element $y-cz$ is in $\Ker f_j\setminus L$ and has support properly 
contained in $\supp y$, a contradiction. Therefore $y$ is the desired 
cycle of minimal support.    

(e) 
Suppose $R$ is a field, and 
$r=\dim_R \bigl(\,\Ker f_j/\im f_{j+1}\bigr)$ is finite for some $j$. 
Then there exist in $\Ker f_j$ cycles of minimal support 
$z_1,\dots, z_{r}$ that induce a basis of $\Ker f_j/\im f_{j+1}$. Indeed, 
if $z_1,\dots,z_{k-1}$ have already been constructed, set 
$L_k=\im f_{j+1}+ (z_1,\dots, z_{k-1})$ and use (d) to pick a cycle of minimal 
support $z_k$ inside $\Ker f_j\setminus L_k$.

(f) 
The \emph{simple syzygies} of 
Charalambous and Thoma \cite{ChThoma1,ChThoma2} 
are a toric analogue for the cycles with minimal support 
of monomial ideals. See Remarks~\ref{R:toric-remarks} for 
more details on the relationship between these two notions.  
\end{remarks}

The following two lemmas are presumably 
familiar to experts, and we include their routine proofs 
for completeness. Throughout all arguments, we assume that   
$R=\Bbbk[x_1,\dots, x_m]$ is a polynomial ring over 
a field $\Bbbk$ with the standard $\mathbb Z^m$-grading.

\begin{lemma}\label{L:dehomogenization}
Suppose $\mathcal F_\bullet$ is a $\mathbb Z^m$-graded minimal 
free reolution of a 
$\mathbb Z^m$-graded $R$-module $M$. Let $B$ be a basis  
of $\mathcal F_\bullet$ consisting of homogeneous elements. 
Let \/ 
$
\overline{\mathcal F}_\bullet = 
\mathcal F_\bullet/(x_1-1,\dots, x_m-1)\mathcal F_\bullet,  
$
let \/ $\ol M=M/(x_1-1,\dots, x_n-1)M$, 
let $\overline B$ be the induced by $B$ basis of 
$\overline{\mathcal F}_\bullet$, and 
let $\bar f_0\: \ol F_0 \lra \ol M$ be the induced by $f_0$ map. 
Then: 

(a) For each $n\ge 0$ and each $\bar b\in \overline B$ 
one has $\bar f_n(\bar b)\ne 0$. 

(b) $\overline B$ is a basis with minimal support if and only if  
$B$ is. 
%
\end{lemma}

\begin{proof} 
Part (a) is immediate from the minimality of the resolution 
$\mathcal F_\bullet$. 
Next, we note the basic fact that 
\begin{equation}\label{Eq:dehomogenization}
\mathcal F_\bullet\otimes_R T \cong 
\overline{\mathcal F}_\bullet\otimes_\Bbbk T
\end{equation}
as chain complexes of 
$\mathbb Z^m$-graded $R$-modules, where $T=R[x_1^{-1},\dots, x_m^{-1}]$ 
is the Laurent polynomial ring over $\Bbbk$ with the standard 
$\mathbb Z^m$-grading, and the elements of 
$
\overline{\mathcal F}_\bullet\otimes 1 
$
are of degree $0$. In particular, an 
element 
$\smash{\bar z=\sum_{\bar b\in \supp(\bar z)} a_b \bar b}$ 
(where each $a_b\in\Bbbk$) 
is in $\smash{\Ker\bar f_n}$  if and only if for some $\gamma\in\mathbb N^m$ 
the homogeneous element 
\[
z= x^\gamma\sum_{\bar b\in\supp(\bar z)} a_bx^{\alpha-\deg b} b
\] 
is in $\Ker f_n$,   
where $\alpha\in\mathbb Z^m$ is the coordinate-wise 
maximum of the degrees $\deg b\in\mathbb Z^m$ with 
$\bar b\in\supp \bar z$. 
Part (b) is now straightforward from the fact 
that for each homogeneous element $y\in F_n$ one has that 
$\overline{\supp_{B} y\vphantom{\bar y}} = \supp_{\overline B}\bar y$.    
\end{proof}

\begin{lemma}\mlabel{L:multigraded-basis-minimal-support}
Let $\mathcal F_\bullet$ be a $\mathbb Z^m$-graded minimal 
free resolution of a 
torsion-free $\mathbb Z^m$-graded 
$R$-module $M$. Let $G$ be a group with a 
free action on $\mathcal F_\bullet$ 
and let $B'$ be a homogeneous basis of $\mathcal F_\bullet$ that is 
preserved under the action of $G$. 
Suppose that within each $G$-orbit of $B'$  
distinct elements have incomparable $\mathbb Z^m$-degrees.  
Let $C$ be any homogeneous basis of $F_0$ that is preserved  
under the action of $G$. 

Then $\mathcal F_\bullet$ has a homogeneous basis  
$B$ that is preserved under the action of $G$, 
has minimal boundary support, and satisfies $B_0=C$. 
\end{lemma}

\begin{proof}
We construct the basis $B_k$ inductively on $k$, with 
the case $k=0$ being the base, where we set $B_0=C$. 
Suppose $B_k$ has already been constructed as desired. 
Let $Z=\Ker f_k$ and  
let $B'_{k+1}$ be our given homogeneous basis of $F_{k+1}$ that 
is preserved under the action of $G$. 
Pick a $G$-orbit $O$ in $B'_{k+1}$, and 
suppose that  for some $b'\in O$ 
its boundary 
is not an element of minimal support in $Z$. 
We will construct iteratively a sequence of homogeneous elements 
$b'=b_0, b_1,\dots, b_n=b''$, all of degree $\deg b'$, such that 
$f_{k+1}(b_n)$ is a cycle of minimal 
support, and for each $1\le i\le n$ the following three conditions hold: 
\begin{enumerate} 
\item 
the set $(B\setminus\{b'\})\cup\{b_i\}$ is 
a homogeneous basis of $F_{k+1}$; 

\item 
we have strict inclusions 
$\supp_{B_k}f_{k+1}(b_i)\subsetneq\supp_{B_k}f_{k+1}(b_{i-1})$; and 

\item  
$\supp_{B'}(b_i)\cap O = \{b'\}$. 
\end{enumerate} 
So, we set $b_0=b'$, and suppose $b_i$ has already been 
constructed as desired. If 
$
z=f_{k+1}(b_i)=\sum_{b\in\supp_{B_k}(z)} a_bx^{\deg z-\deg b} b
$ 
is already a cycle of minimal support 
then we have $n=i$ and we are done with our sequence. Otherwise, 
there is a homogeneous element $0\ne z'\in Z$ 
with $\supp_{B_k}(z')\subsetneq\supp_{B_k}(z)$, and 
we can construct $b_{i+1}$ as follows. Let 
$z'=\sum_{b\in\supp_{B_k}(z')} a'_bx^{\deg z'-\deg b} b$. Since $F_k/Z$ is 
torsion-free, we can assume in addition that $\deg z'$ is 
the coordinate-wise maximum of the degrees $\deg b\in\mathbb Z^m$ 
with $b\in\supp_{B_k}(z')$. In particular, 
$\deg z'\le \deg z =\deg b_i=\deg b'$. 
Since $\mathcal F_\bullet$ is a resolution of $M$, there is 
a homogeneous $w=\sum_{b\in\supp_{B'}(w)}c_b b\in F_{k+1}$ 
such that $\deg w=\deg z'$ and $f_{k+1}(w)=z'$. Notice that 
$\supp_{B'}(w)$ consists only of elements of degrees $\le\deg b'$, thus it  
can contain at most one element from the orbit $O$, 
namely $b'$ itself.  
If $b'\in \supp_{B'}(w)$ then we must have 
$\deg w=\deg b'$, therefore $0\ne c_{b'}\in \Bbbk$, and we can 
set $b_{i+1}=w$. If not, then 
$b'\notin\supp_{B'}(w)$, and selecting some $b\in\supp_{B'}(w)$ 
we can set $b_{i+1}=a'_b b_i - a_bx^{\deg b' -\deg z'}w$. 

Now that our sequence has been constructed, we have 
in our hands the homogeneous element $b''$ of same degree as
$b'$, with $f_{k+1}(b'')$ a cycle of minimal support,  
and, furthermore, the intersection of   
$\supp_{B'}(b'')$ with the orbit $O$ is exactly $b'$.  
Therefore we can use the free 
action of $G$ and replace in $B'_{k+1}$ all elements of $O$ by 
the elements of the orbit $O'$ of $b''$, thus obtaining 
a homogeneous basis $B''_{k+1}$ of $F_{k+1}$ preserved under the action of $G$.  

Since the $G$-orbits of $B''_{k+1}$ are obtained from the $G$-orbits of 
$B'_{k+1}$ by replacing $O$ with $O'$, and all 
elements of $O'$ have boundaries that are cycles of minimal support,  
iterating the previous arguments over all orbits of $G$ in $B_{k+1}'$     
yields the desired homogeneous basis with minimal support $B_{k+1}$ and 
completes our induction argument.  
\end{proof}

\begin{example}\mlabel{E:TwoResolutions}
Although bases of minimal support always exist, 
they are not unique. Consider the monomial ideal 
$M=(vy, wz, vwx, xyz)$ in the 
polynomial ring $R=\Bbbk[v,w,x,y,z]$ over a field $\Bbbk$. 
Due to degree constraints, all homogeneous basis elements 
for the free modules 
in a minimal $\mathbb Z^5$-graded free resolution 
$\mathcal F_\bullet$ of $M$ 
are uniquely determined up to associates, 
except the basis elements corresponding to 
the direct summand $R(-\mbf{1})^2$ that appears in 
homological dimension two of the resolution (in general, we write  
$\mbf{1}=(1^m)=(1,\dots,1,\dots,1)\in\mathbb{Z}^m$).  
Under one choice of basis, we have a basis element \ $e$ \ 
(corresponding to the rightmost column of the 
matrix $\partial_2$ displayed below)
which is \emph{not} with minimal boundary support:  
\begin{equation*}
   \begin{CD} 
    {R^4} 
	@< {
            \begin{bsmallmatrix*}[r]
             -wx & 0 &  -wz & -xz & 0 \\
              0 &  -vx & vy &  0 &  -xy \\
              y  & z &  0  & 0 &  0   \\
              0 &  0 &  0 &  v  & w   \\
            \end{bsmallmatrix*}
	} <\partial_1<
    {R^{5}} 
	@< {
               \begin{bsmallmatrix*}[r]
               -z & -z  \\
               y & y\\
               x & 2x \\
               0 & -w  \\
               0 & v \\
            \end{bsmallmatrix*}
	} <\partial_2< 
    {R^{2}} \leftarrow 0. 
   \end{CD}
\end{equation*}
This is because the cycle $z$ given by the left column of the 
matrix for $\partial_2$ has support strictly contained in the support 
of $\partial_2(e)$. Illustrating the proof of 
Lemma~\ref{L:multigraded-basis-minimal-support}, to get a basis of 
minimal support we can proceed in two ways: by considering 
$\partial_2(e)-z$ or $\partial_2(e)-2z$.  
With the first way 
we arrive at a basis element (corresponding to the rightmost column) 
with minimal boundary support 
consisting of three elements: 
\begin{equation*}
   \begin{CD} 
    {R^4} 
	@< {
            \begin{bsmallmatrix*}[r]
             -wx & 0 &  -wz & -xz & 0 \\
              0 &  -vx & vy &  0 &  -xy \\
              y  & z &  0  & 0 &  0   \\
              0 &  0 &  0 &  v  & w   \\
            \end{bsmallmatrix*}
	} <\partial_1<
    {R^{5}} 
	@< {
               \begin{bsmallmatrix*}[r]
               -z & 0  \\
               y & 0 \\
               x & x \\
               0 & -w  \\
               0 & v \\
            \end{bsmallmatrix*}
	} <\partial_2< 
    {R^{2}} \leftarrow 0. 
   \end{CD}
\end{equation*}
Note that this basis choice for $\mathcal F_\bullet$ 
corresponds to the homogenization of the conic chain complex described 
in Example~\ref{E:NonCWPoset}. 
With the second way we arrive at 
a basis element (corresponding to the rightmost column) 
with minimal boundary support 
consisting of four elements: 
\begin{equation*}
   \begin{CD} 
    {R^4} 
	@< {
            \begin{bsmallmatrix*}[r]
             -wx & 0 &  -wz & -xz & 0 \\
              0 &  -vx & vy &  0 &  -xy \\
              y  & z &  0  & 0 &  0   \\
              0 &  0 &  0 &  v  & w   \\
            \end{bsmallmatrix*}
	} <\partial_1<
    {R^{5}} 
	@< {
               \begin{bsmallmatrix*}[r]
               -z & z  \\
               y & -y \\
               x & 0 \\
               0 & w  \\
               0 & -v \\
            \end{bsmallmatrix*}
	} <\partial_2< 
    {R^{2}} \leftarrow 0. 
   \end{CD}
\end{equation*}
We therefore have two substantially different choices 
of a basis, each having minimal support, for $\mathcal F_\bullet$. 
\end{example}

\begin{remark}\mlabel{R:RigidUniqueBasis}
The case where there is (up to associates) 
only one homogeneous basis (and hence no choices arise) 
is naturally of 
interest. Recall that   
a monomial ideal is \emph{rigid}, see e.g. \cite{ClMa}, if the following 
two conditions on its $\mathbb{Z}^m$-graded Betti numbers hold:
\begin{itemize}
\item[(R1)] 
$\beta_{i,\alpha}$ is either 1 or 0 for all $i$ 
and all $\alpha\in\mathbb{Z}^m$; and

\item[(R2)] 
If $\beta_{i,\alpha} = 1$ and $\beta_{i,\alpha'} = 1$ 
then $\alpha$ and $\alpha'$ are incomparable in $\mathbb Z^m$.
\end{itemize}
By \cite[Proposition~1.5]{ClMa} a minimal resolution of a monomial  
ideal has a unique up to associates $\mathbb{Z}^m$-graded basis 
if and only if the ideal is rigid. 
In particular, the unique $\mathbb{Z}^m$-graded basis in a minimal 
resolution of a rigid ideal is necessarily a basis with minimal support. 
\end{remark}

\begin{example}\mlabel{E:PPRes}
Let $S=\Bbbk[x_1,\ldots,x_6]$ with $\Bbbk$ a field of 
characteristic 2, and consider the triangulation $\nabla$ 
of the projective plane 
(see Bruns and Herzog \cite[p. 236]{BrHe})   
with Stanley-Reisner ideal 
{\small
\[
I=
( x_1x_2x_3,  x_1x_3x_5,  x_1x_4x_5, 
 x_2x_3x_4,  x_2x_4x_5, 
 x_1x_2x_6,  x_1x_4x_6, 
 x_2x_5x_6,  x_3x_4x_6,  x_3x_5x_6 ).
\]
}
\negmedspace
\negthinspace
The ideal 
$I$ has total Betti numbers $(10,15,7,1)$. 
The choice of basis $B$ for a minimal free resolution 
$\mathcal F_\bullet$ of $I$ as  
given by Macaulay2 \cite{M2} yields 
{\scriptsize
\begin{equation*}
   \begin{CD} 
   {S^{10}}
     @< {
               \begin{bsmallmatrix*}[c]
               x_4 & x_5 & 0 & 0 & 0 & x_6 & 0 & 0 & 0 & 0 & 0 & 0 & 0 & 0 & 0 \\
                x_1 & 0 & 0 & 0 & x_5 & 0 & 0 & 0 & x_6 & 0 & 0 & 0 & 0 & 0 & 0 \\
                0 & x_2 & 0 & x_4 & 0 & 0 & 0 & 0 & 0 & 0 & x_6 & 0 & 0 & 0 & 0 \\
                0 & 0 & x_2 & x_3 & 0 & 0 & 0 & 0 & 0 & 0 & 0 & 0 & x_6 & 0 & 0 \\
                0 & 0 & x_1 & 0 & x_3 & 0 & 0 & 0 & 0 & 0 & 0 & 0 & 0 & x_6 & 0 \\
                0 & 0 & 0 & 0 & 0 & x_3 & x_4 & 0 & 0 & x_5 & 0 & 0 & 0 & 0 & 0 \\
                0 & 0 & 0 & 0 & 0 & 0 & x_2 & x_3 & 0 & 0 & 0 & 0 & x_5 & 0 & 0 \\
                0 & 0 & 0 & 0 & 0 & 0 & 0 & x_1 & x_2 & 0 & 0 & 0 & 0 & 0 & x_5 \\
                0 & 0 & 0 & 0 & 0 & 0 & 0 & 0 & 0 & x_1 & 0 & x_3 & 0 & x_4 & 0 \\
                0 & 0 & 0 & 0 & 0 & 0 & 0 & 0 & 0 & 0 & x_1 & x_2 & 0 & 0 & x_4 
            \end{bsmallmatrix*}
  } <d_1< 
    {S^{15}} 
	@< {
               \begin{bsmallmatrix*}[c]
                x_5 & x_6 & 0 & 0 & 0 & 0 & 0  \\
                 x_4 & 0 & x_6 & 0 & 0 & 0 & x_4x_6 \\
                 x_3 & 0 & 0 & x_6 & 0 & 0 & 0  \\
                 x_2 & 0 & 0 & 0 & x_6 & 0 & x_2x_6 \\
                 x_1 & 0 & 0 & 0 & 0 & x_6 & 0  \\
                 0 & x_4 & x_5 & 0 & 0 & 0 & x_4x_5 \\
                 0 & x_3 & 0 & x_5 & 0 & 0 & x_3x_5 \\
                 0 & x_2 & 0 & 0 & x_5 & 0 & 0  \\
                 0 & x_1 & 0 & 0 & 0 & x_5 & 0  \\
                 0 & 0 & x_3 & x_4 & 0 & 0 & 0  \\
                 0 & 0 & x_2 & 0 & x_4 & 0 & 0  \\
                 0 & 0 & x_1 & 0 & 0 & x_4 & 0  \\
                 0 & 0 & 0 & x_2 & x_3 & 0 & x_2x_3 \\
                 0 & 0 & 0 & x_1 & 0 & x_3 & 0  \\
                 0 & 0 & 0 & 0 & x_1 & x_2 & 0  
            \end{bsmallmatrix*}
	} <d_2< 
    {S^{7}} 
	@< 
	{
	\begin{bsmallmatrix*}[c]
		x_6 \\
               	x_5  \\ 
               	x_4 \\
                 x_3 \\
                 x_2 \\
                 x_1 \\
                 0 
            \end{bsmallmatrix*}
	} <d_3< 
    {S,}
   \end{CD}
\end{equation*}
}
\negthickspace
and it is routine to check that this is indeed a 
basis of minimal support. 
\end{example}

\begin{remark}
In our experience, Macaulay2 \cite{M2} has always produced a basis of minimal 
support in its calculation of minimal free resolutions of monomial ideals. 
It would be interesting to find a monomial ideal where this was not the case. 
\end{remark}

\section{Incidence posets}\mlabel{S:Incidence}

Let the ring $R$, the nonnegative chain complex $\mathcal F_\bullet$, 
the map $f_0\: F_0\lra M$, and the basis $B$ be 
as in the 
beginning of the previous section; in particular we do not 
assume that $\mathcal F_\bullet$ is acyclic, and we 
do not assume that $B$ has minimal boundary support. 
For each $b\in B_n$ with $n\ge 1$ we can write 
\[
f_n(b)=\sum_{c\in B_{n-1}}a_{b,c}\, c,  
\]  
and we call the uniquely determined coefficient 
$a_{b,c}\in R$ the 
\emph{incidence coefficient} of $b$ and $c$ (with respect 
to the chosen bases $B_n$ and $B_{n-1}$).

\begin{definition}\mlabel{D:incidence-poset}
We introduce a poset structure  
$P(\mathcal F_\bullet, B)$ on the set $B$ 
by taking the partial ordering 
generated by the relations $b>c \iff a_{b,c}\ne 0$. We call 
this poset the \emph{incidence poset} of $\mathcal F_\bullet$ 
(with respect to the chosen basis $B$).   
\end{definition}

\begin{remarks}\mlabel{R:ranked-poset} 
Let  $P=P(\mathcal F_\bullet, B)$. 

(a) 
$\mathcal F_\bullet$ is naturally $P$-graded with 
$F_{n,a}= Ra$ for $a\in B_n$, 
and $F_{n,a}=0$ otherwise. 

(b) 
For each $a\in B_n$ with $n\ge 1$ we have \ 
$\supp_B f_n(a)=\{b\in B_{n-1}\mid b<a\}$. 

(c) 
Suppose $f_n(b)\ne 0$ for each $n\ge 1$ and each $b\in B_n$. 
It is a straightforward consequence of the definition that 
for every $n\ge 0$ and every $b\in B_n$ we have 
\[
d(b)=\dim \Delta(P_{\le b}) = n,
\] 
and  when $a\le b$ every maximal strictly increasing  
chain in $P$ that starts with $a$ and ends with $b$ has length 
$d(b) - d(a)$. In particular, the set $B_0$ is the set of minimal elements 
of $P$.  
\end{remarks}

\begin{example}\label{E:CWincidence}
Let $X$ be a CW-complex, and let $C_\bullet(X,\Bbbk)$ be its 
cellular chain complex with coefficients in $\Bbbk$. Let $A_i$ be 
the basis of $C_i(X,\Bbbk)=\HH_i(X^i,X^{i-1},\Bbbk)$ given by the 
homology classes of the characteristic maps of the $i$-cells of $X$, 
see Remark~\ref{R:conic-not-cellular}. 
Let $A=\coprod A_i$. Then the incidence poset 
of $C_\bullet(X,\Bbbk)$ with respect to $A$ is exactly the face 
poset $P=P(X)$, and the $P$-grading on $C_\bullet(X,\Bbbk)$ from 
Remark~\ref{R:ranked-poset}(a) coincides with the $P$-grading on 
$C_\bullet(X,\Bbbk)$ from Example~\ref{E:P-grading-examples}(c).  
\end{example}

\begin{example}\mlabel{E:PPincidence} 
The incidence poset $P(\mathcal F_\bullet,B)$ for the homogeneous basis $B$  
of a minimal resolution $\mathcal F_\bullet$ of the Stanley-Reisner ideal 
$I$ of the projective plane triangulation $\nabla$ 
presented in Example~\ref{E:PPRes} has 
the following 
Hasse diagram:  
\begin{center}
 \begin{tikzpicture}[scale=0.75, 
                     vertices/.style={draw, fill=black, circle, inner sep=1pt},
                     vertices2/.style={draw,fill=white,circle,inner sep=2pt}]
             \node [vertices] (1) at (-9/2+0,2){};
             \node [vertices] (2) at (-9/2+1,2){};
             \node [vertices] (20) at (-9/2+2,2){};
             \node [vertices] (4) at (-9/2+3,2){};
             \node [vertices] (5) at (-9/2+4,2){};
             \node [vertices] (8) at (-9/2+5,2){};
             \node [vertices] (25) at (-9/2+6,2){};
             \node [vertices] (12) at (-9/2+7,2){};
             \node [vertices] (13) at (-9/2+8,2){};
             \node [vertices] (15) at (-9/2+9,2){};
             \node [vertices] (3) at (-7+0,4){};
             \node [vertices] (21) at (-7+1,4){};
             \node [vertices] (7) at (-7+2,4){};
             \node [vertices] (6) at (-7+3,4){};
             \node [vertices] (9) at (-7+4,4){};
             \node [vertices] (10) at (-7+5,4){};
             \node [vertices] (29) at (-7+6,4){};
             \node [vertices] (26) at (-7+7,4){};
             \node [vertices] (17) at (-7+8,4){};
             \node [vertices] (22) at (-7+9,4){};
             \node [vertices] (14) at (-7+10,4){};
             \node [vertices] (23) at (-7+11,4){};
             \node [vertices] (16) at (-7+12,4){};
             \node [vertices] (18) at (-7+13,4){};
             \node [vertices] (27) at (-7+14,4){};
             \node [vertices2] (33) at (-4+9,6){};
             \node [vertices] (11) at (-4+0,6){};
             \node [vertices] (30) at (-4+1.5,6){};
             \node [vertices] (24) at (-4+3,6){};
             \node [vertices] (19) at (-4+4.5,6){};
             \node [vertices] (28) at (-4+6,6){};
             \node [vertices] (31) at (-4+7.5,6){};
             \node [vertices2] (32) at (-0+0,8){};
     \foreach \to/\from in {1/16, 1/7, 1/3, 2/9, 2/17, 2/3, 3/11, 3/19, 4/21, 4/6, 4/7, 5/10, 5/6, 5/23, 6/24, 6/11, 7/28, 7/11, 8/29, 8/9, 8/10, 9/30, 9/11, 10/11, 10/31, 11/32, 12/17, 12/14, 12/22, 13/14, 13/18, 13/23, 14/24, 14/19, 15/16, 15/18, 15/27, 16/28, 16/19, 17/30, 17/19, 18/19, 18/31, 19/32, 20/21, 20/26, 20/22, 21/24, 21/28, 22/24, 22/30, 23/24, 23/31, 24/32, 25/29, 25/26, 25/27, 26/28, 26/30, 27/28, 27/31, 28/32, 29/30, 29/31, 30/32, 31/32, 26/33, 22/33, 14/33, 18/33, 27/33}
     \draw [-] (\to)--(\from);
     \end{tikzpicture}
\end{center}   
Note that the open lower set generated by the maximal 
element of dimension $3$ is precisely the face poset of 
the simplicial complex $\nabla$, thus the order complex of 
this lower set is isomorphic to the first barycentric subdivision of 
$\nabla$. 
\end{example}

\begin{example}\mlabel{E:TwoIncidence}
In general, different choices of bases with minimal support 
for the same chain complex will result in non-isomorphic 
incidence posets. For instance, 
the two choices for a basis with minimal support 
from Example~\ref{E:TwoResolutions} 
give rise to two non-isomorphic incidence posets 
whose Hasse diagrams are displayed below.  
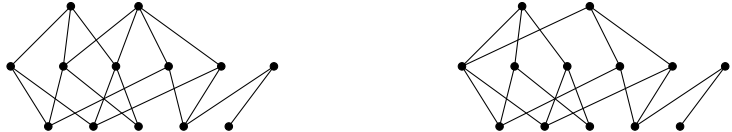
\begin{center}
\begin{tikzpicture}[scale=0.9]
\begin{scope}[scale=1, shift={(-3.7,0)}, vertices/.style={draw, fill=black, circle, inner sep=1pt}]
             \node [vertices, label=right:{}] (1) at (-2.25+0,1.33333){};
             \node [vertices, label=right:{}] (2) at (-2.25+1.5,1.33333){};
             \node [vertices, label=right:{}] (3) at (-2.25+3,1.33333){};
             \node [vertices, label=right:{}] (4) at (-2.25+4.5,1.33333){};
             \node [vertices, label=right:{}] (12) at (-3+0,2.66667){};
             \node [vertices, label=right:{}] (13) at (-3+1.5,2.66667){};
             \node [vertices, label=right:{}] (23) at (-3+3,2.66667){};
             \node [vertices, label=right:{}] (24) at (-3+4.5,2.66667){};
             \node [vertices, label=right:{}] (34) at (-3+6,2.66667){};
             \node [vertices, label=right:{}] (123) at (-1.5,4){};
             \node [vertices, label=right:{}] (234) at (1.5,4){};
     \foreach \to/\from in {1/12, 1/13, 2/12, 2/23, 2/24, 3/13, 3/23, 3/34, 4/24, 4/34, 12/123, 13/123, 23/123, 23/234, 24/234, 34/234}
     \draw [-] (\to)--(\from);
\end{scope}

\begin{scope}[scale=1, shift={(3.7,0)}, vertices/.style={draw, fill=black, circle, inner sep=1pt}]
             \node [vertices, label=right:{}] (1) at (-2.25+0,1.33333){};
             \node [vertices, label=right:{}] (2) at (-2.25+1.5,1.33333){};
             \node [vertices, label=right:{}] (3) at (-2.25+3,1.33333){};
             \node [vertices, label=right:{}] (4) at (-2.25+4.5,1.33333){};
             \node [vertices, label=right:{}] (12) at (-3+0,2.66667){};
             \node [vertices, label=right:{}] (13) at (-3+1.5,2.66667){};
             \node [vertices, label=right:{}] (23) at (-3+3,2.66667){};
             \node [vertices, label=right:{}] (24) at (-3+4.5,2.66667){};
             \node [vertices, label=right:{}] (34) at (-3+6,2.66667){};
             \node [vertices, label=right:{}] (123) at (-1.5,4){};
             \node [vertices, label=right:{}] (1234) at (1.5,4){};
     \foreach \to/\from in {1/12, 1/13, 2/12, 2/23, 2/24, 3/13, 3/23, 3/34, 4/24, 4/34, 12/123, 13/123, 23/123, 24/1234, 34/1234,12/1234, 13/1234}
     \draw [-] (\to)--(\from);
\end{scope}

\end{tikzpicture}
\end{center}
\end{example}

The following theorem provides a key technical tool needed 
for the proofs of our main results. 

\begin{theorem}\mlabel{T:conic-iso} 
Let $\Bbbk$ be a field, 
and let $B$ be a $\Bbbk$-basis 
for a  nonnegative chain complex 
$\mathcal F_\bullet$ of \/ $\Bbbk$-modules 
with $\dim_\Bbbk\HH_0\mathcal F_\bullet = 1$. 
Let $f_0\:F_0\lra \HH_0\mathcal F_\bullet$ be the canonical 
projection. 
Let $P=P(\mathcal F_\bullet, B)$ and suppose 
for each $n\ge 0$ that $f_n(b)\ne 0$ for all \/    
$b\in B_n$. 
The following are equivalent: 
\begin{enumerate}
\item 
$B$ is a basis with minimal support; 

\item 
For each $n\ge 0$ and \ $a\in B_n$ we have 
$\dim_\Bbbk\HH_n(\Delta(P_{\le a}), \Delta(P_{<a}), \Bbbk)=1$,   
and there is an isomorphism of $P$-graded chain complexes 
\[
\phi_\bullet\: 
\mathcal F_\bullet \lra 
\mathcal C_\bullet(P, \Bbbk),   
\]  
i.e. such that 
$\phi_n(a)\in \HH_n(\Delta(P_{\le a}), \Delta(P_{<a}), \Bbbk)$ 
for all $n\ge 0$ and $a\in B_n$. 
\end{enumerate} 
\end{theorem}

\begin{proof} 
Note that our assumptions   
imply, by Remark~\ref{R:ranked-poset}(c), 
that the set of minimal elements of $P$ is precisely 
$B_0$, and, more generally, that  
$B_n=\{a\in B \mid d(a)=n\}$. 
Thus we have  
\[
\mathcal C_n(P, \Bbbk) = 
\HH_n(\Delta^n,\Delta^{n-1}, \Bbbk) = 
\bigoplus_{a\in B_n}
\HH_n\bigl(\Delta(P_{\le a}),\Delta(P_{<a}), \Bbbk\bigr). 
\] 
Furthermore, for each $a\in B_n$ 
we have by Remark~\ref{R:augmented-conic}(c) a canonical inclusion 
$
\mathcal C_\bullet(P_{\le a},\Bbbk)\subseteq 
\mathcal C_\bullet(P,\Bbbk), 
$
hence Proposition~\ref{P:conic-form} shows that  
the differential $\partial_n$ of 
the conic chain complex $\mathcal C_\bullet(P, \Bbbk)$ maps  
$\HH_n\bigl(\Delta(P_{\le a}), \Delta(P_{<a}), \Bbbk\bigr)$ 
isomorphically onto  
\[
K^a_{n-1}:=\RH_{n-1}(\Delta(P_{< a}), \Bbbk)
\]
when $n\ge 1$;  and 
by \eqref{E:conic-kernel} and Remarks~\ref{R:augmented-conic}(a,c) this 
is exactly the kernel of the restriction of $\partial_{n-1}$ to 
\[
C^a_{n-1}:= 
\bigoplus_{
a>b\in B_{n-1} 
}
\HH_{n-1}\bigl(\Delta(P_{\le b}),\Delta(P_{<b}), \Bbbk\bigr) 
\quad = \quad \mathcal C_{n-1}(P_{\le a}, \Bbbk).
\]
We also set 
\[
F^a_{n-1} := 
\bigoplus_{a> b\in B_{n-1}}
\Bbbk\thinspace b
\]
and write $Z^a_{n-1}$ for the kernel of the restriction of 
$f_{n-1}$ to $F^a_{n-1}$. 

We now show that $(2)$ implies $(1)$. Indeed, suppose 
$B$ is not with minimal support. Then for some $n\ge 1$ 
and some $a\in B_n$ there is a non-zero 
cycle $c\in F_{n-1}$ with $\supp c \subsetneq\supp f_n(a)$. 
Therefore by Remarks~\ref{R:ranked-poset}(b) and 
\ref{R:change-basis}(b) the space $Z^a_{n-1}$ 
is at least two-dimensional. Since 
$\phi_\bullet$ maps $F^a_{n-1}$ isomorphically onto $C^a_{n-1}$, this 
yields that $Z^a_{n-1}$ is isomorphic to $K^a_{n-1}$ and hence the latter 
has to be also at least two-dimensional, a contradiction. 

Next we proceed with the proof that $(1)$ implies $(2)$. 
Since $\dim_\Bbbk\HH_0\mathcal F_\bullet=1$, we will fix 
a generator $m\ne 0$ of $\HH_0\mathcal F_\bullet$ over $\Bbbk$, and 
by Remark~\ref{R:change-basis}(a) we may assume 
without loss of generality that $f_0(b)=m$ 
for each $b\in B_0$.  We will also write 
$
\phi_{-1}\: \HH_0\mathcal F_\bullet \lra 
\mathcal C_{-1}(P,\Bbbk)=\Bbbk\cdot[\ ] 
$
for the isomorphism that sends $m$ to $[\ ]$. 
By using induction on $n\ge 0$
we will simultaneously construct the desired isomorphisms $\phi_n$ and 
prove that 
\begin{equation}\elabel{E:top-dimension-one}
\dim_\Bbbk\HH_n(\Delta(P_{\le a}), \Delta(P_{<a}), \Bbbk)=1
\end{equation}
for all $a\in B_n$. 
Note that for each $a\in B_0$ we have 
$\HH_0\bigl(\Delta_{\le a},\Delta_{<a},\Bbbk)=\Bbbk\cdot [a]$ 
by Proposition~\ref{P:conic-elements}, thus 
\eqref{E:top-dimension-one} holds 
when $n=0$ and  
$
\HH_0(\Delta^0,\Delta^{-1}, \Bbbk)= 
\bigoplus_{a\in B_0} \Bbbk\cdot [a].
$
Since 
$F_0=\bigoplus_{a\in B_0} \Bbbk\thinspace a$, 
we define $\phi_0$ 
by $\phi_0(a)=[a]$ and observe that 
$\partial_0\phi_0= \phi_{-1} f_0$.  
This establishes the base of our induction. 

Now assume $n\ge 1$ and that for all 
$0\le k\le n-1$ we have already 
established \eqref{E:top-dimension-one} 
and have constructed  
isomorphisms $\phi_k$ 
as required in (2), and such that 
$\partial_k\phi_k= \phi_{k-1} f_k$.
In particular, $F^a_{n-1}$ gets mapped by $\phi_{n-1}$ 
isomorphically onto $C^a_{n-1}$, and therefore 
$Z^a_{n-1}$ gets mapped by $\phi_{n-1}$ isomorphically 
onto $K^a_{n-1}$.  
The 
minimal support condition on $B$ 
implies, by Remarks~\ref{R:change-basis}(b) 
and~\ref{R:ranked-poset}(b),   
that $Z^a_{n-1}$ has dimension $1$ and is spanned by $f(a)$. 
Therefore $K^a_{n-1}$ has dimension $1$, 
and is spanned by 
$\phi_{n-1}(f_n(a))$.  This completes the proof of 
\eqref{E:top-dimension-one}, and we can now define 
$\phi_n$ by sending each $a\in B_n$ to the 
unique nonzero element in 
$\HH_n\bigl(\Delta_{\le a}, \Delta_{<a}, \Bbbk\bigr)$ 
that gets mapped under $\partial_n$ to $\phi_{n-1}(f_n(a))$. 
The remaining 
desired properties of $\phi_n$ are now immediate.   
\end{proof}

\begin{example}\label{E:regular-CW-basis-minimal-support}
Let $\Bbbk$ be a field, $X$ be a regular CW-complex, and $A$ be the basis of 
$C_\bullet(X,\Bbbk)$ given by the homology classes of the 
characteristic maps of the cells of $X$, 
see Example~\ref{E:CWincidence}. Then $A$ is a basis with minimal 
boundary support. Indeed, since the closure of each cell 
involves only finitely many other cells, we may without 
loss of generality assume  that $X$ is finite and connected. 
Then $C_\bullet(X,\Bbbk)$ is isomorphic as a $P$-graded chain 
complex to $\mathcal C_\bullet(P,\Bbbk)$, see 
Example~\ref{E:cellular-is-conic}, where $P$ is the face 
poset of $X$. Since 
$P$ is an hcw-poset by \cite{Bjo}, the desired conclusion 
is immediate from Theorem~\ref{T:conic-iso}.   
\end{example}

\begin{example}\label{E:nearly-scarf}
Let $\Omega$ be a finite simplicial complex, and let $\Bbbk$ 
be a field. 

(a) 
By Example~\ref{E:regular-CW-basis-minimal-support}, the 
standard basis $A=\coprod A_j$ of $C_\bullet(\Omega,\Bbbk)$ 
is a basis of minimal boundary support. For each $j\ge 0$  
let $r_j=\dim_\Bbbk\RH_j(\Omega, \Bbbk)$, and let $z_{j,1},\dots,z_{j,r_j}$ be 
cycles of minimal support that induce a basis of $\RH_j(\Omega,\Bbbk)$; 
these exist by Remark~\ref{R:change-basis}(e). Let 
$A'_0=A_0$ and $C'_0=C_0(\Omega,\Bbbk)$; and for $j\ge 0$ set 
$A'_{j+1}=A_{j+1}\sqcup\{e_{j+1,1},\dots,e_{j+1,r_j}\}$ and 
\[
C'_{j+1}\ = \ 
C_{j+1}(\Omega,\Bbbk)\ \oplus \ 
\Bbbk\cdot e_{j+1,1}\ \oplus\ \dots\ \oplus\ \Bbbk\cdot e_{j+1,r_j}. 
\] 
Define $d'_{j+1}\: C'_{j+1}\lra C'_j$ by $d'_{j+1}(e_{j+1,k})=z_{j,k}$ 
and $d'_{j+1}(a)=d_{j+1}(a)$ for $a\in A_{j+1}$, where $d_{j+1}$ 
is the differential of $C_\bullet(\Omega,\Bbbk)$. Thus we obtain 
a nonnegative chain complex $C'_\bullet=(C_j',d_j')$ such that 
$\dim_\Bbbk\HH_0C'_\bullet=1$, 
together with a basis of minimal support $A'=\coprod A'_j$. 
Let $P=P(C'_\bullet, A')$ be the corresponding incidence poset. 
Then $C'_\bullet$ is isomorphic as a $P$-graded chain complex 
to $\mathcal C_\bullet(P,\Bbbk)$ by Theorem~\ref{T:conic-iso}. 
Furthermore $P$ is obtained from the face poset of $\Omega$ by adding  
$r_j$ new elements of dimension $j+1$, for $j= 0,\dots,\dim\Omega$, 
representing cycles of minimal support that induce a basis 
of $\RH_j(\Omega,\Bbbk)$, with 
new covering 
relations determined by their supports.  

(b) 
Let $I_\Omega$ be the Nearly-Scarf ideal \cite{PeVe} associated 
with $\Omega$. As shown in \cite{PeVe} there is a labeling 
$\deg$ on $P$ such that the homogenization of $C'_\bullet$ is a minimal 
free resolution of $I_\Omega$. It follows that a minimal resolution of  
$I_\Omega$ is supported on the labeled poset $(P,\deg)$. Furthermore, 
$P$ is obtained from the lcm-lattice of $I_\Omega$ by removing 
the minimal and the maximal elements, and adding elements corresponding 
to cycles of  minimal support that induce a basis for 
$\RH_\bullet(\Omega,\Bbbk)$, with new covering relations determined by 
the supports of these cycles.  
\end{example}

We are now ready to address 
the first main result of this paper, namely  
that every minimal free resolution of a monomial ideal 
is supported on a poset. We will state and prove it 
in the more general setting of monomial 
modules. Let $\Bbbk$ be a field, let $R=\Bbbk[x_1,\dots, x_m]$ 
be a polynomial ring over $\Bbbk$, and consider the Laurent 
polynomials ring $T=\Bbbk[x_1,\dots, x_m,x_1^{-1},\dots,x_m^{-1}]$ 
with the standard $\mathbb Z^m$-grading.  
Recall from \cite{BaSt} that a \emph{monomial module} is a 
$\mathbb Z^m$-graded $R$-submodule of $T$.

\begin{theorem}\mlabel{T:MFR-from-poset}
Let\/ $\Bbbk$ be a field, let $R=\Bbbk[x_1,\dots, x_m]$ 
be a polynomial ring over $\Bbbk$, let\/ $M$ be a 
monomial module in $T$, and let\/ $\mathcal F_\bullet$ be 
a minimal\/ $\mathbb Z^m$-graded free resolution of\/ $M$ over $R$. 
Suppose that\/ $B$ is a homogeneous basis of\/ $\mathcal F_\bullet$ 
with minimal support, and let \/ 
$\deg\: P(\mathcal F_\bullet, B) \lra \mathbb Z^m$ be the map that 
assigns to each element of \/ $B$ its\/ $\mathbb Z^m$-degree as an 
element of \/ $\mathcal F_\bullet$.
 
Then\/ $\deg$  is a morphism of posets, and the conic chain complex 
$\mathcal C_\bullet\bigl(P(\mathcal F_\bullet, B),\Bbbk\bigr)$ 
produces a minimal free resolution of\/ 
$M$ after homogenization. In particular, 
every monomial ideal has a minimal free resolution 
supported on a poset. 
\end{theorem}

\begin{proof}
That $\deg$ is a morphism of posets is immediate from 
the fact that the differential of $\mathcal F_\bullet$ preserves 
the $\mathbb Z^m$-grading.  

Let $P=P(\mathcal F_\bullet, B)$, let 
$
\ol{\mathcal F}_\bullet=
\mathcal F_\bullet/(x_1-1,\dots, x_m-1)\mathcal F_\bullet,
$ 
and let $\ol B$ be the induced by $B$ basis of 
$\ol{\mathcal F}_\bullet$. 
It is straightforward to observe that we have 
$P(\mathcal F_\bullet, B)= P(\ol{\mathcal F}_\bullet, \ol B)$. 
Also, by Lemma~\ref{L:dehomogenization} we have that 
$\overline{B}$ is a basis of minimal support, and that 
$\bar f_n(\bar b)\ne 0$ for all $n\ge 0$ and all $b\in B_n$. 
Furthermore, the isomorphism \eqref{Eq:dehomogenization} shows 
that $\overline{\mathcal F}_\bullet$ is acyclic, hence a 
resolution of $\overline M=M/(x_1-1,\dots,x_m-1)M\cong\Bbbk$. 
Thus 
by Theorem~\ref{T:conic-iso} we have an isomorphism 
$
\phi_\bullet \: \ 
\ol{\mathcal F}_\bullet \lra \mathcal C_\bullet(P,\Bbbk)
$
of $P$-graded chain complexes
which therefore lifts to an isomorphism of the corresponding 
homogenized chain complexes. Since the homogenization of 
$\ol{\mathcal F}_\bullet$ is $\mathcal F_\bullet$, we obtain the 
desired conclusion.  
\end{proof}

\begin{remark} 
In view of Theorem~\ref{T:MFR-from-poset} it becomes an 
important open problem 
to construct (preferably in a canonical way) a poset $P$ that 
supports a minimal free resolution of a given monomial ideal $I$, 
without already knowing the structure of that resolution. 
One possible approach to this problem is to start with the 
lcm-lattice of the ideal $I$ and find a 
way to modify it appropriately. 
If one doesn't insist on a canonical 
construction, then one can proceed in a manner somewhat 
analogous to what we did in Example~\ref{E:nearly-scarf}, 
as follows. Let $L$ be the 
lcm-lattice with its minimal element removed. 
The order complex $\Delta(L)$ supports a non-minimal free resolution 
of $I$ (this is called the lcm-resolution of $I$ in \cite{OrWe}). Then, 
starting with $n=2$,  
one can begin adding new elements to $L$ in such a way so that   
their dimension is $<n$ and 
the order complex of the new poset still supports a resolution of $I$, 
but open lower sets generated by elements of dimensions $d\le n$ 
have zero reduced homology in dimensions $\le d-2$. As we keep increasing $n$, 
we eventually reach a poset $P$ that satisfies the 
assumptions of Proposition~\ref{P:conic-simplicial}, and whose 
order complex $\Delta(P)$ 
still supports a free resolution of $I$. Now one can use 
Proposition~\ref{P:conic-simplicial} and deduce that the poset $P$ 
supports a minimal free resolution of $I$.  
Due to dimensional restrictions, essentially the smallest 
nontrivial example of this procedure arises for the monomial ideal 
$I$ from Example~\ref{E:PPRes}. In that case the Hasse diagram 
of its lcm-lattice (without the minimal element) is obtained by removing 
from the incidence poset displayed in Example~\ref{E:PPincidence}
the maximal element of dimension $2$, and (after making appropriate choices) 
the new poset that gets produced (after one iteration only)  
is the poset displayed in Example~\ref{E:MakePPHCW}.    
Unfortunately, as 
this extremely technical 
iterative process involves making numerous non-canonical 
choices at each stage, it is very difficult to control, and 
in general the properties of the resulting poset $P$ are 
still shrouded in mystery.    
\end{remark}

\section{Homology CW-posets}\mlabel{S:hcw}

One of the main goals in this paper is to show that our notion of resolution 
supported on a poset allows to extend to every minimal resolution 
many useful techniques originally developed to study cellular 
resolutions. For that we need labeled posets whose elements are in a 
degree-preserving bijective correspondence with a homogeneous basis 
of a minimal resolution, \emph{and} 
whose topological properties reflect well the properties of the resolution. 
From such a point of view the lcm-lattice of a monomial ideal $I$ is, 
in general, not an appropriate poset to use, even though 
its topology determines completely the structure of the minimal resolutions  
of the ideal. This is because the lcm-lattice may contain elements 
that either do not contribute, or contribute more than once 
to the corresponding graded Betti number of the ideal.  
Similarly, the incidence poset $P=P(\mathcal F_\bullet, B)$ for a basis 
with minimal support $B$ of a minimal free resolution $\mathcal F_\bullet$ 
of $I$ may not be appropriate to use, even though its elements are in 
a natural bijective correspondence with a homogeneous basis of 
$\mathcal F_\bullet$;  this is because the topology of $P$ may not reflect 
well enough the properties of the resolution. In the remaining three 
sections of this paper we will show that a good 
class of posets that satisfies all desired requirements is  the class 
of hcw-posets.    

Recall that a  poset $P$ is called a 
\emph{homology CW-poset} or \emph{hcw-poset} over $\Bbbk$  
if for each $a\in P$ the dimension of $a$ is finite and the lower set   
$P_{<a}$ is a homology sphere over $\Bbbk$.  

\begin{examples}\mlabel{E:hcw}
(a) 
As already mentioned in Examples~\ref{E:cellular-is-conic} 
and~\ref{E:regular-CW-basis-minimal-support}, by \cite{Bjo} 
every CW-poset is an hcw-poset. Thus hcw-posets  provide 
a good starting point for extending cellular resolution 
techniques.

(b) 
It is straightforward to check that the two incidence posets 
from Example~\ref{E:TwoIncidence} are hcw-posets. 

(c) 
In general one cannot expect that incidence posets 
will be hcw-posets.  For instance,  
the open lower set generated by the dimension $3$ element 
in the incidence poset from Example~\ref{E:PPincidence} 
is homeomorphic 
to the barycentric subdivision of the projective plane, hence  
has nonzero homology in two dimensions. 
This incidence poset 
is therefore not an hcw-poset. 

(d)
Routine dimensional considerations show that, 
for a monomial ideal $I$ and a homogeneous 
basis with minimal support $B$ 
of a minimal free $\mathbb Z^m$-graded resolution 
$\mathcal F_\bullet$ of $I$, the  incidence poset 
$P=P(\mathcal F_\bullet, B)$ has to be an hcw-poset 
if $\dim\Delta(P) \le 2$.   
Thus Example~\ref{E:PPincidence} describes   
essentially the smallest such incidence poset which  
is not an hcw-poset. 
\end{examples}

Next, we check that hcw-posets have sufficiently good 
combinatorial and topological properties. 
It is immediate from the definitions that 
if a labeled hcw-poset supports a free resolution then its elements 
are in a degree-preserving bijective correspondence with a homogeneous 
basis of the resolution. This takes care of the combinatorial 
part. The following proposition indicates that we also have 
reasonably good topological properties: 

\begin{proposition}\label{P:good-hcw-topology} 
Let $\Bbbk$ be a field, let $R=\Bbbk[x_1,\dots,x_m]$ be a polynomial 
ring over $\Bbbk$ with the standard $\mathbb Z^m$-grading, 
and let $M$ be a $\mathbb Z^m$-graded $R$-module. 
Let $Z=(P,\deg)$ be a labeled hcw-poset over $\Bbbk$.  

Then $Z$ supports a free resolution of $M$ if and only 
the simplicial complex $\Delta(P)$ supports a free resolution of $M$.  
\end{proposition} 

\begin{proof} 
The canonical inclusion 
$\mathcal C_\bullet(P,\Bbbk)\lra C_\bullet\bigl(\Delta(P),\Bbbk\bigr)$ of $P$-graded chain complexes lifts to a map between their homogenizations 
$\mathcal C_\bullet(Z,\Bbbk)\lra C_\bullet\bigl(\Delta(Z),\Bbbk\bigr)$ such 
that for each $\alpha\in\mathbb Z^m$ the corresponding map between 
the degree $\alpha$ strands can be canonically identified with the 
inclusion map 
$
\mathcal C_\bullet(P_{\deg\le\alpha},\Bbbk)
\lra 
C_\bullet\bigl(\Delta(P_{\deg\le\alpha}),\Bbbk\bigr).
$ 
Since lower sets of hcw-posets are again hcw-posets, $P_{\deg\le\alpha}$ is 
an hcw-poset for each $\alpha\in\mathbb Z^m$. 
The desired conclusion is now immediate  
from Proposition~\ref{P:conic-simplicial} and  
Proposition~\ref{acyclicity-criterion}. 
\end{proof} 

The next step is to show in our second main theorem  
that given a minimal free resolution $\mathcal F_\bullet$ of a 
monomial ideal we can always find an hcw-poset that supports it. 
We will formally state and prove this result 
in the next section. The main idea is to start with an incidence 
poset for a basis with minimal support of $\mathcal F_\bullet$, 
and modify it appropriately to produce the desired hcw-poset. 
The following quite technical lemma 
and theorem provide  
the key tool for accomplishing that goal. 

\begin{lemma}\mlabel{L:cavity-conic}
Let $n\ge 0$, let $\deg\: B\lra \mathbb Z^m$ be a function, 
and let $P$ be a poset structure on the set $B$ such that 
all elements of $B$ have finite dimension, and such that 
$(P,\deg)$ is a labeled poset. 
Let $a\in B$ be such that $n+2\le d_P(a)$, 
let $\alpha=\deg(a)$, 
and suppose that for all $b\in P$ with $d_P(b)<d_P(a)$ 
the lower set $P_{<b}$ is a homology sphere over $\Bbbk$. 
Also, suppose that 
$\HH_n\widetilde{\mathcal C}_\bullet(P_{\deg\le\alpha},\Bbbk)=0$. 
Then there exists a poset structure $P'$ on $B$ such that 
the following conditions are satisfied:  
\begin{enumerate}
\item 
The maps $\id_B\: P\lra P'$ and $\deg\: P'\lra \mathbb Z^m$ 
are morphisms of posets. 


\item
For each $c\in B$ such that $a\not\le_P c$ we have 
$P_{\le c} = P'_{\le c}$. 

\item
For each $c\in B$ we have 
$\dim\Delta(P_{\le c})=\dim\Delta(P'_{\le c})$. 

\item
$
\widetilde{\mathcal C}_\bullet(P,\Bbbk)=
\widetilde{\mathcal C}_\bullet(P',\Bbbk).
$  

\item
For $k\ge n+1$ we have 
$\RH_k\bigl(\Delta(P_{<a}), \Bbbk\bigr)= 
 \RH_k\bigl(\Delta(P'_{<a}), \Bbbk\bigr)$. 

\item 
$\RH_n\bigl(\Delta(P'_{<a}), \Bbbk\bigr) = 0$. 
\end{enumerate}
In fact, the poset $P'$ is obtained by 
creating extra edges extending down from $a$ in the 
Hasse diagram of $P$. 
\end{lemma}

\begin{proof}
Let $Y$ be a basis  
of $\RH_n\bigl(\Delta(P_{<a}),\Bbbk\bigr)$, and choose a well-ordering on $Y$.  

The first key observation is that when $m\le n$ every 
element $g$ of $\RH_m(\Delta(P_{<a}), \Bbbk)$ can be  
represented by an $m$-cycle $z$ of the form 
\begin{equation}\elabel{E:reduced-form}
z= \sum_{c\in A_z}
[c, z_c], 
\end{equation}
where $A_z$ is an antichain in $P_{<a}$ with 
$d(c)=m$ for each $c\in A_z$, and 
each $(m-1)$-chain $z_c$ is a cycle of 
$\Delta(P_{<c})$. 
Indeed, any $m$-chain $w$ of $\Delta(P)$ 
can be uniquely written in the form 
\begin{equation}\elabel{E:cone-form}
w = \sum_{c\in A_w} [c, w_c], 
\end{equation}
for some subset $A_w$ of $P$ and
$(m-1)$-chains $w_c$ of $\Delta(P_{<c})$. Let 
$k=\max\{ d(c) \mid c\in A_w\}$, and let 
$c_1,\dots, c_l$ be the elements of $A_w$ of 
dimension $k$. If $w$ is a cycle, this forces each 
$w_{c_i}$ to be a cycle.  
If in addition $k>m$, then by our assumptions $w_{c_i}$ is 
also a boundary of $\Delta( P_{<c_i})$, thus for each 
$i$ we have a chain $v_i$ of $\Delta( P_{<c_i})$ such 
that $\partial(v_i)=w_{c_i}$. Therefore, if 
$k>m$ and $w$ represents  
$g$ then $A_w\subseteq P_{<a}$ and the cycle 
\[
w'= w + \partial\left(\sum_{i=1}^m [c_i, v_i] \right)
\] 
also represents $g$ and has 
$\max\{d(c) \mid c\in A_{w'}\}< k$. Iterating this 
procedure we arrive at the desired cycle $z$. 

Second, we note that, considered as a cycle of 
$\Delta(P)$, a cycle $z$ of the form \eqref{E:reduced-form} 
that represents an element $h\in Y$ is an element of 
$\widetilde{\mathcal C}_\bullet(P_{\deg\le\alpha}, \Bbbk)$. 
Since we also assume 
$\HH_n\widetilde{\mathcal C}_\bullet(P_{\deg\le\alpha}, \Bbbk)=0$, 
this yields an $(n+1)$-chain $w=w(h)$ in 
$\widetilde{\mathcal C}_\bullet(P_{\deg\le\alpha},\Bbbk)$ such that 
$\partial(w)=z$, in particular 
by Proposition~\ref{P:conic-elements} we have $d(c)=n+1$ 
and $w_c$ is a cycle for each $c\in A_w$. 
Note that while $A_w\subseteq P_{\deg\le\alpha}$, we also have 
$A_w\not\subset P_{<a}$. 

Third, using transfinite induction, for each $h\in Y$ 
select a chain $w=w(h)$ as above and such that the (finite) set  
$C_h=A_w\setminus (P_{<a}\cup\bigcup_{g<h}C_g)$ is minimal. 
Set $C=\bigcup_{h\in Y}C_h$ and make a new poset structure 
$P'$ on the set $B$ by adding to the relations from $P$ 
the new relations $c<a$ for $c\in C$ and taking the poset 
they generate. 
It is now straightforward to check that conditions 
(1) through (3) of the Lemma are satisfied for $P'$. 
Furthermore 
since in $P'_{<a}$ the elements of $C$ are maximal but 
of dimension $n+1< d(a)$, we get that 
$\RH_k(\Delta(P_{<a}), \Bbbk)=\RH_k(\Delta(P'_{<a}), \Bbbk)$ 
for $k\ge n+2$. 

Next, 
we observe that any $(n+1)$-cycle 
$z$ of $\Delta(P'_{<a})$ in the form 
\eqref{E:reduced-form}
is in fact a cycle of $\Delta(P_{<a})$. Indeed, if  
$A_z\cap C$ is nonempty, then there is a maximal $h\in Y$ such that 
there exists $c\in A_z\cap C_h$. Then for $w=w(h)$ we get that 
$z_c$ and $w_c$ are both $n$-cycles in 
the $n$-dimensional homology sphere 
$\Delta(P'_{<c})=\Delta(P_{<c})$ hence are the same up to 
a scalar multiple $s\in\Bbbk$. Therefore the chain 
$w-sz$ contradicts the minimality of the choices of $w(h)$ and $C_h$. 
It follows that 
$
\RH_{n+1}(\Delta(P'_{<a}),\Bbbk)=
\RH_{n+1}(\Delta(P_{<a}),\Bbbk).
$
Thus conditions (4) and (5) of the Lemma are also satisfied. 

Finally, every element of $\RH_n(\Delta(P'_{<a}),\Bbbk)$ 
can be represented by a cycle of the form 
\eqref{E:reduced-form}, and all these are also cycles of 
$\Delta(P_{<a})$. Therefore the map $\id_B\: P \lra P'$ 
induces a surjection 
$
\RH_{n}(\Delta(P_{<a}),\Bbbk)\lra 
\RH_{n}(\Delta(P'_{<a}),\Bbbk)
$  
that has, by construction, every element $h\in Y$ in its kernel. 
\end{proof}



\begin{theorem}\mlabel{T:hcw-from-inc}
Let $\deg\: B\lra \mathbb Z^m$ be a function. 
Let $P$ be a poset structure on the set $B$ 
such that $(P,\deg)$ is a labeled poset, 
for each $a\in B$ 
one has \ $d_P(a)<\infty$ \ and \ 
$\smash{\dim_\Bbbk \RH_{d_P(a)-1}(\Delta(P_{<a}), \Bbbk)=1}$,  
and such that 
$\smash{\widetilde{\mathcal C}_\bullet(P_{\deg\le\alpha}, \Bbbk)}$ 
is exact whenever $P_{\deg\le\alpha}$ is non-empty. 
Suppose also that a group $G$ acts on $P$. 
Then there exists an hcw-poset structure $Q$ on the set 
$B$ such that 
\begin{enumerate}
\item 
$\id_B\: P\lra Q$ and $\deg\: Q\lra\mathbb Z^m$ 
are morphisms of posets;  

\item
for each $a\in B$ we have 
$\dim \Delta(P_{\le a})=\dim \Delta(Q_{\le a})$; 

\item
the action of $G$ on $P$ is also an action on $Q$; and 

\item
$
\widetilde{\mathcal C}_\bullet(P,\Bbbk)=
\widetilde{\mathcal C}_\bullet(Q,\Bbbk).
$  
\end{enumerate}
\end{theorem}

\begin{proof} 
Note that the elements in each $G$-orbit $O$ have the same dimension, 
which we denote by $d_P(O)$. Now we apply the following two-step process 
repeatedly to all orbits $O$, in increasing order of $d_P(O)$: 
\begin{enumerate}  
\item 
Pick $a\in O$ and use Lemma~\ref{L:cavity-conic} repeatedly with   
$n = d(a)-2, d(a)-3,\dots, 0$, arriving at a poset structure 
where the open lower set generated by $a$ is a homology sphere; and then  

\item
Use the action of $G$ and propagate the new covering relations  
involving $a$ obtained in step (1) to produce new covering relations 
involving the remaining elements of the orbit $O$. 
\end{enumerate} 
At the end of this iterative process we arrive at the desired 
hcw-poset $Q$. 
\end{proof}

\begin{remark}
If $P_1=P(\mathcal F_\bullet, B')$ and 
$P_2=P(\mathcal F_\bullet, B'')$ are two non-isomorphic 
incidence posets for the same minimal free resolution $\mathcal F_\bullet$ 
of a monomial ideal $I$, then in general 
any hcw-posets $Q_1$ and $Q_2$ produced from them by 
applying the procedure from the proof of Theorem~\ref{T:hcw-from-inc}
will also be non-isomorphic. A trivial example for this is when 
the incidence posets are already hcw, like the two posets from 
Example~\ref{E:TwoIncidence}. 
\end{remark}

\begin{example}\label{E:MakePPHCW}
Applying the procedure from the proof of 
Theorem~\ref{T:hcw-from-inc} to 
the incidence poset $P(\mathcal F_\bullet, B)$ from 
Example~\ref{E:hcw}(c) produces a hcw-poset 
with Hasse diagram obtained by adding 
the single new dashed edge 
to the solid edges of the Hasse diagram from 
Example~\ref{E:PPincidence}:  
\begin{center}
 \begin{tikzpicture}[scale=0.75, 
  vertices/.style={draw, fill=black, circle, inner sep=1pt},
  vertices2/.style={draw,fill=white,circle,inner sep=2pt}]
             \node [vertices] (1) at (-9/2+0,2){};
             \node [vertices] (2) at (-9/2+1,2){};
             \node [vertices] (20) at (-9/2+2,2){};
             \node [vertices] (4) at (-9/2+3,2){};
             \node [vertices] (5) at (-9/2+4,2){};
             \node [vertices] (8) at (-9/2+5,2){};
             \node [vertices] (25) at (-9/2+6,2){};
             \node [vertices] (12) at (-9/2+7,2){};
             \node [vertices] (13) at (-9/2+8,2){};
             \node [vertices] (15) at (-9/2+9,2){};
             \node [vertices] (3) at (-7+0,4){};
             \node [vertices] (21) at (-7+1,4){};
             \node [vertices] (7) at (-7+2,4){};
             \node [vertices] (6) at (-7+3,4){};
             \node [vertices] (9) at (-7+4,4){};
             \node [vertices] (10) at (-7+5,4){};
             \node [vertices] (29) at (-7+6,4){};
             \node [vertices] (26) at (-7+7,4){};
             \node [vertices] (17) at (-7+8,4){};
             \node [vertices] (22) at (-7+9,4){};
             \node [vertices] (14) at (-7+10,4){};
             \node [vertices] (23) at (-7+11,4){};
             \node [vertices] (16) at (-7+12,4){};
             \node [vertices] (18) at (-7+13,4){};
             \node [vertices] (27) at (-7+14,4){};
             \node [vertices2] (33) at (-4+9,6){};
             \node [vertices] (11) at (-4+0,6){};
             \node [vertices] (30) at (-4+1.5,6){};
             \node [vertices] (24) at (-4+3,6){};
             \node [vertices] (19) at (-4+4.5,6){};
             \node [vertices] (28) at (-4+6,6){};
             \node [vertices] (31) at (-4+7.5,6){};
             \node [vertices2] (32) at (-0+0,8){};
     \foreach \to/\from in {1/16, 1/7, 1/3, 2/9, 2/17, 2/3, 3/11, 3/19, 4/21, 4/6, 4/7, 5/10, 5/6, 5/23, 6/24, 6/11, 7/28, 7/11, 8/29, 8/9, 8/10, 9/30, 9/11, 10/11, 10/31, 11/32, 12/17, 12/14, 12/22, 13/14, 13/18, 13/23, 14/24, 14/19, 15/16, 15/18, 15/27, 16/28, 16/19, 17/30, 17/19, 18/19, 18/31, 19/32, 20/21, 20/26, 20/22, 21/24, 21/28, 22/24, 22/30, 23/24, 23/31, 24/32, 25/29, 25/26, 25/27, 26/28, 26/30, 27/28, 27/31, 28/32, 29/30, 29/31, 30/32, 31/32, 26/33, 22/33, 14/33, 18/33, 27/33}
     \draw [-] (\to)--(\from);
     \draw [-,dashed] (32)--(33);
     \end{tikzpicture}
\end{center}  
\end{example}

\section{Resolutions supported on hcw-posets}\mlabel{S:Main}

We are now ready to prove the second main result of this paper: 

\begin{theorem}\mlabel{T:cl-supports-I}
Let $\Bbbk$ be a field, and 
let \thinspace $I$ be a monomial ideal in 
$R=\Bbbk[x_1,\dots, x_m]$. 

There exists a labeled hcw-poset \thinspace $(Q,\deg)$ 
that supports a minimal free $\mathbb Z^m$-graded resolution of 
\thinspace $I$ over $R$. 
\end{theorem}

\begin{proof}
Let $\mathcal F_\bullet$ be a minimal free $\mathbb Z^m$-graded 
resolution of $I$ over $R$. By 
Lemma~\ref{L:multigraded-basis-minimal-support} there exists a 
homogeneous basis $B$ of minimal support for $\mathcal F_\bullet$. 
Let $\deg\: B \lra \mathbb Z^m$ be the map that assigns to each 
element of $B$ its $\mathbb Z^m$-degree as an element of 
$\mathcal F_\bullet$.  By Theorem~\ref{T:MFR-from-poset}, for  
the incidence poset $P=P(\mathcal F_\bullet, B)$ we have that 
$\deg\: P \lra \mathbb Z^m$ is a morphism of posets,  
and the conic chain complex $\mathcal C_\bullet(P, \Bbbk)$ 
produces (up to isomorphism) after homogenization 
the resolution $\mathcal F_\bullet$.  In particular, the chain 
complex $\smash{\widetilde{\mathcal C}_\bullet(P_{\deg\le\alpha},\Bbbk)}$ 
is exact whenever $P_{\deg\le\alpha}$ is nonempty.  
By Remark~\ref{R:ranked-poset} we have $d(a)=n$ for every $a\in B_n$, 
hence Theorem~\ref{T:conic-iso} yields that 
$\smash{\dim_\Bbbk\RH_{d(a)-1}(\Delta(P_{<a}),\Bbbk)=1}$ for each $a\in P$. 
Therefore Theorem~\ref{T:hcw-from-inc} produces the desired  
hcw-poset structure $Q$ on the set $B$. 
\end{proof}

\begin{example}\label{E:PPHCWSuppRes}
As the proof of Theorem~\ref{T:cl-supports-I} indicates, 
the hcw-poset created in Example~\ref{E:MakePPHCW} 
supports a minimal free resolution of the Stanley-Reisner ideal 
$I$ from Example~\ref{E:PPRes}. 
\end{example}

In general, the 
isomorphism class of a hcw-poset $Q$ that supports a minimal 
free resolution of a monomial ideal $I$ is not unique. 
When $I$ belongs to the class of rigid ideals, as observed in 
Remark~\ref{R:RigidUniqueBasis} 
there is a unique choice for a basis of minimal support, and 
the degree constraints on the basis elements yield that 
the corresponding incidence poset is exactly the Betti poset of 
$I$ over $\Bbbk$. (The Betti poset of a 
monomial ideal $I$ over $\Bbbk$ is the set 
of the $\mathbb Z^m$-degrees 
of the elements in a homogeneous basis of a minimal free 
$\mathbb Z^m$-graded resolution of $I$. See \cite{ClMa,ClMa2,TchVa} for more 
on Betti posets.)  
It turns out that for rigid ideals the Betti poset is already 
an hcw-poset, and that this property characterizes  
rigid ideals:

\begin{theorem}\mlabel{T:BettihcwiffRigid}
The Betti poset of a monomial ideal is an hcw-poset 
if and only if the ideal is rigid. 
\end{theorem}

\begin{proof}
Let $P$ be the Betti poset of a monomial ideal $I$, and  
suppose first $P$  
is an hcw-poset. Rigidity 
condition (R1) is automatically satisfied, since 
$\beta_{a,d(a)-1}=1$ holds for all $a\in P$. If condition 
(R2) did not hold, then we could find two comparable 
elements $a < b$ in $\mathbb Z^m$ where 
$\beta_{a,i}=\beta_{b,i}=1$ for some $i$. 
Under this assumption, the comparability $a< b$ 
would also hold in $P$. However, $a < b$ implies 
that $d(a)<d(b)$. Since $P$ is an hcw-poset  
then $\RH_{d(a)-1}(\Delta(P_{<a}),\Bbbk)\cong\Bbbk$ 
and $\RH_{d(b)-1}(\Delta(P_{<b}),\Bbbk)\cong\Bbbk$ 
are forced, a contradiction.  
Thus, $I$ is rigid.   

Suppose now that $I$ is a rigid monomial ideal. 
For $a\in P$, we will show that $P_{<a}$ is a homology sphere by 
induction on $d=d(a)$. 
If $d=0$ then $\Delta(P_{<a})$ is a non-empty simplicial complex 
when $a$ is not the degree of a minimal generator of $I$. Thus, 
$\RH_{-1}(\Delta(P_{<a}),\Bbbk)=0$. On the other hand, 
when $a$ is the degree of a minimal generator of $I$, 
the poset $P_{<a}$ is empty, and we have 
$\RH_{-1}(\Delta(P_{<a}),\Bbbk)\cong\Bbbk$ as desired. 
Suppose $d\ge 1$ and that for all $b\in P$ with $d(b)< d$, 
we have the desired isomorphisms  
$\RH_{d(b)-1}(\Delta(P_{<b}),\Bbbk)\cong\Bbbk$ and  
$\RH_i(\Delta(P_{<b}),\Bbbk)=0$ for $i\ne d(b)-1$. 
If the element $a\in P$ has $d(a)=d$, 
then since $I$ is rigid, there exists $i$ such that 
$\RH_j(\Delta(P_{<a}),\Bbbk)$ is isomorphic to $\Bbbk$ 
when $j=i$, and vanishes for $j\ne i$. 
We must show that $i=d-1$. We certainly have 
$i\le d-1$ since $P_{<a}$ is $(d-1)$-dimensional. 
Aiming for a contradiction, suppose $i<d-1$. 
The structure of the order complex of $P_{<a}$ guarantees that there 
exists $b< a\in P$ with $d(b)=i+1$. 
Using the fact that $b\in P$, 
we have $\RH_i(\Delta(P_{<b}),\Bbbk)\cong\Bbbk$. 
However, $b < a$ and  
both having nonzero homology in the same dimension $i$ 
contradicts rigidity condition (R2). Thus $i<d-1$ is impossible, 
making $P$ an hcw-poset. 
\end{proof}

We are now all set to demonstrate how techniques developed to study 
cellular resolutions can be applied to resolutions supported on 
posets. As our first example of this process we show how to apply 
methods from \cite{BaSt} to prove that minimal resolutions of 
toric rings are supported on posets.  

Let $L\subseteq \mathbb Z^m$ be a lattice such that $L\cap\mathbb N^m=0$ and 
such that $\mathbb Z^m/L$ is a torsion-free abelian group. 
Let $Q$ be the image of $\mathbb N^m$ under the canonical projection 
$\mathbb Z^m \lra \mathbb Z^m/L$, and for $\alpha\in\mathbb Z^m$ let  
$\bar\alpha$ be its image in $\mathbb Z^m/L$. This 
induces a $(\mathbb Z^m/L)$-grading on the 
polynomial ring $R=\Bbbk[\mathbb N^m]=\Bbbk[x_1,\dots,x_m]$ and 
a structure of a $(\mathbb Z^m/L)$-graded $R$-module on the \emph{toric ring} 
$S=\Bbbk[Q]$. We are interested in the structure of a minimal free 
resolution of $S$ as a $(\mathbb Z^m/L)$-graded $R$-module. To this end, 
let $R[L]$ be the group ring of $L$ over $R$. The inclusion 
$L\subseteq \mathbb Z^m$ together with the standard $\mathbb Z^m$-grading 
on $R$ induce a $\mathbb Z^m$-grading on $R[L]$.  
We will 
exploit the equivalence $\pi$ from the category of $\mathbb Z^m$-graded 
$R[L]$-modules to the category of $(\mathbb Z^m/L)$-graded $R$-modules 
that was established by Bayer and Sturmfels~\cite{BaSt}. 
We recall from \cite{BaSt} 
that if $N$ is a $\mathbb Z^m$-graded $R[L]$-module, then 
$\pi N=N\otimes_{R[L]}(R[L]/J)$ where $J$ is  the 
augmentation ideal generated in 
$R[L]$ by the set $\{1- l \mid l\in L\}$. In particular, when 
$M_L$ is the monomial $R$-module generated over $R$ 
by the elements of $L$ inside the Laurent polynomial 
ring $T=\Bbbk[\mathbb Z^m]=\Bbbk[x_1,\dots,x_m,x_1^{-1},\dots,x_m^{-1}]$, 
then $M_L$ is also a $\mathbb Z^m$-graded $R[L]$-module 
and $\pi M_L\cong S$.   

To state our results, we need to introduce the relevant terminology. 

\begin{definition}
A \emph{toric poset} is a $\mathbb Z^m$-labeled poset $X=(P,\deg)$ together 
with a free action of $L$ on $P$ such that for all $\alpha\in L$ 
and $a\in P$ one has $\deg(\alpha a)=\alpha + \deg(a)$. 
\end{definition}

Let $X=(P,\deg)$ be a toric poset. The $L$-action on $P$ induces 
a free $L$-action on the conic chain complex $\mathcal C_\bullet(X,\Bbbk)$, 
thus giving it the structure of a chain complex of free $\mathbb Z^m$-graded 
$R[L]$-modules. Therefore the equivalence $\pi$ yields a chain complex 
$\pi\mathcal C_\bullet(X,\Bbbk)$ of free $(\mathbb Z^m/L)$-graded $R$-modules. 

\begin{definition} 
We say that a chain complex $\mathcal F_\bullet$ of 
$(\mathbb Z^m/L)$-graded $R$-modules 
is \emph{supported on a toric poset} $X$ if   $\mathcal F_\bullet$ is 
isomorphic to the chain complex $\pi\mathcal C_\bullet(X,\Bbbk)$. 
\end{definition}

The following is the third main result of this paper:

\begin{theorem}\label{T:toric-resolutons-are-hcw} 
There exists a toric hcw-poset that supports a 
minimal $(\mathbb Z^m/L)$-graded free resolution of the toric ring 
$S$ over $R$. 
\end{theorem}

\begin{proof} 
We recall some additional facts from 
\cite{BaSt}. As part of their proof of \cite[Theorem 3.2]{BaSt}, 
Bayer and Sturmfels construct, for each $(\mathbb Z^m/L)$-graded 
$R$-module $N=\bigoplus_{\bar\alpha\in\mathbb Z^m/L}N_{\bar\alpha}$, a 
$\mathbb Z^m$-graded $R[L]$-module 
$\tau N=M=\bigoplus_{\alpha\in\mathbb Z^m}M_{\alpha}$ such that $\pi M = N$. 
More specifically, one has $(\tau N)_\alpha=M_\alpha=N_{\bar\alpha}$ for 
each $\alpha\in\mathbb Z^m$, multiplication by the variable $x_i$ 
is the map $M_{\alpha}\rightarrow M_{\alpha+e_i}$ that is given by the 
map $x_i\: N_{\bar\alpha}\rightarrow N_{\bar\alpha+\bar e_i}$, and  
multiplication by $l\in L$ is the identity map on 
$M_\alpha = N_{\bar\alpha} = N_{\overline{\alpha+l}}=M_{\alpha + l}$. In particular, 
it is immediate from this construction that if $F$ is a free 
$(\mathbb Z^m/L)$-graded $R$-module then $\tau F$ is a free 
$\mathbb Z^m$-graded $R[L]$-module. Furthermore, since $\pi$ is 
fully faithful, for each homogeneous morphism of $(\mathbb Z^m/L)$-graded 
$R$-modules $f\:N' \rightarrow N''$ there is a unique homogeneous morphism    
$\tau f\: \tau N' \rightarrow \tau N''$ of $\mathbb Z^m$-graded 
$R[L]$-modules such that $\pi\tau f =f$. In fact, as shown in the proof 
of \cite[Theorem 3.2]{BaSt}, for each $\alpha\in\mathbb Z^m$ the 
map $(\tau f)_\alpha\: (\tau N')_\alpha\rightarrow (\tau N'')_\alpha$ equals 
the map $f_{\bar\alpha}\: N'_{\bar\alpha}\rightarrow N''_{\bar\alpha}$. 
Therefore $\tau$ is an exact functor that for any 
$(\mathbb Z^m/L)$-graded $R$-free resolution $\mathcal F_\bullet$ yields 
a $\mathbb Z^m$-graded $R[L]$-free resolution $\tau\mathcal F_\bullet$.   
In addition, since $S$ has a homogeneous basis over $\Bbbk$ that 
is preserved under multiplication by each $x_i$, we see  that 
$\tau S$ has a homogeneous basis over $\Bbbk$ that is preserved under 
multiplication by each $x_i$ and by each $l\in L$. As the homogeneous 
components of $S$ are $1$-dimensional in degrees $\bar\alpha\in Q$ and 
zero otherwise, we see  that the homogeneous components of $\tau S$ 
are $1$-dimensional in degrees $\alpha\in \mathbb N^m + L$ and zero 
otherwise. Thus $\tau S$ is isomorphic to the monomial 
module $M_L$. 
 
%
Now let $\mathcal F_\bullet$ be a minimal free resolution of $S$. 
Then $\mathcal G_\bullet=\tau\mathcal F_\bullet$ is a free resolution 
of $M_L$ with $\pi\mathcal G_\bullet=\mathcal F_\bullet$, 
hence is minimal by \cite[Corollary 3.3]{BaSt}. It follows 
that $\mathcal G_\bullet$ is also minimal as a  
$\mathbb Z^m$-graded $R$-free resolution of $M_L$,  
and $L$ acts 
freely on it so that for each homogeneous element $x$ and each 
$\alpha\in L$ we have $\deg(\alpha x)=\alpha + \deg(x)$. 
It suffices to show that $\mathcal G_\bullet$ is supported on a 
toric hcw-poset. 
Note that, by Lemma~\ref{L:multigraded-basis-minimal-support}, 
the complex $\mathcal G_\bullet$ 
has a homogeneous $R$-basis $B$ with minimal boundary support  
and such that $B$ is preserved under the action of $L$.  
Let $P$ be the corresponding incidence poset. Thus $L$ 
acts freely on $P$, and therefore 
the $\mathbb Z^m$-labeled poset $X=(P,\deg)$ is a toric poset. 
Also, by Theorem~\ref{T:MFR-from-poset} the resolution 
$\mathcal G_\bullet$ is supported on $X$.  The desired 
conclusion is now immediate from Theorem~\ref{T:hcw-from-inc}. 
\end{proof} 

It is natural to ask whether one can obtain a toric poset 
that supports a minimal $(\mathbb Z^m/L)$-graded free resolution 
$\mathcal F_\bullet$ of $S$ over $R$ directly from an appropriately 
chosen basis of $\mathcal F_\bullet$. We address this and some other 
interesting facts about posets and toric resolutions 
in the remarks below. 
Since they are not needed in the rest of the paper, we leave their 
reasonably straightforward proofs as exercises for the interested 
reader.

\begin{remarks}\label{R:toric-remarks} 
Let $\mathcal F_\bullet=(F_n,f_n)$ be a minimal $(\mathbb Z^m/L)$-graded 
free resolution of the toric ring $S$ over $R$, and let 
$\mathcal G_\bullet=\tau\mathcal F_\bullet$ be the corresponding 
minimal $\mathbb Z^m$-graded free resolution of $M_L$ over $R[L]$.   
Let $B_n$ be a homogeneous basis of the free $R$-module $F_n$,  
and let $B=\coprod_n B_n$.  

(a) 
For each $\alpha\in\mathbb Z^m$ 
the degree $\bar\alpha$ strand 
$(\mathcal F_{\bullet})_{\bar\alpha}$ 
is isomorphic to the degree $\alpha$ strand $(\mathcal G_\bullet)_\alpha$.  

(b) 
If $(P,\deg)$ is a toric poset that supports  
$\mathcal F_\bullet$ then for each $\alpha\in\mathbb Z^m$ 
the degree $\bar\alpha$ strand 
$(\mathcal F_{\bullet})_{\bar\alpha}$ 
is isomorphic to 
$\mathcal C_\bullet(P_{\deg\le\alpha}, \Bbbk)$. 

(c) 
Let $\nu\: \mathbb Z^m/L \lra \mathbb Z^m$ be a 
homomorphism of abelian groups that is a splitting of the 
quotient map $\mathbb Z^m\lra \mathbb Z^m/L$. Using $\nu$, we consider 
a degree $\bar\alpha$ element $b\in B_n$ as a degree $\nu(\bar\alpha)$ 
element of $\tau F_n$ via the identification 
$(\tau F_n)_{\nu(\bar\alpha)}=(F_n)_{\bar\alpha}$. In this way $B$ is  
identified also as a subset of $\mathcal G_\bullet$ and it is 
straightforward to check that this makes $B$ a homogeneous 
basis of $\mathcal G_\bullet$ over $R[L]$. Therefore 
the set 
\[
L\times B=\{\lambda b \mid b\in B, \ \lambda\in L\}
\] 
is a homogeneous 
basis for $\mathcal G_\bullet$ when considered as a minimal 
$\mathbb Z^m$-graded free resolution of $M_L$ over $R$.  

(d)
Let $n\ge 1$. For any $b\in B_n$ of 
$\mathbb Z^m$-degree $\deg(b)=\alpha$ with 
\[
f_n(b)=
\sum_{
e\in B_{n-1},\, \gamma\in\mathbb N^m 
}
c_{e,\gamma}\, x^\gamma e
\] 
(where each $c_{e,\gamma}\in\Bbbk$)
and any $\mu\in L$ we have
\[
(\tau f_n)(\mu b)=
\sum_{
e\in B_{n-1},\, \gamma\in\mathbb N^m 
}
c_{e,\gamma}\ x^\gamma (\lambda_\gamma+\mu) e, 
\] 
where $\lambda_\gamma:=\nu(\bar\gamma) - \gamma\in L$. 
In particular, $x^\gamma e$ is in the $B$-support of $f_n(b)$ as defined 
in \cite{ChThoma1,ChThoma2} if and only if 
for some (and hence for each) $\mu\in L$ the  
boundary support of $\mu b$ with respect to 
the basis $L\times B$, as per our  
Definition~\ref{D:basis-minimal-support}, 
contains $(\lambda_\gamma + \mu)e$. 

(e) 
The based resolution $(\mathcal F_\bullet, B)$ is a 
\emph{simple resolution} in the 
sense of \cite{ChThoma1,ChThoma2} if and only if $L\times B$ is a 
basis of minimal support for $\mathcal G_\bullet$ over $R$. 

(f) 
The labeled incidence poset 
$\bigl(P(\mathcal G_\bullet, L\times B), \deg\bigr)$ is a toric 
poset, and its isomorphism class does not depend on the choice 
of the splitting $\nu$. Thus we denote it as $P(\mathcal F_\bullet, B)$ 
and call it the \emph{toric incidence poset} of $\mathcal F_\bullet$ 
with respect to $B$. 

(g)
If $(\mathcal F_\bullet, B)$ is a simple resolution then 
$\mathcal F_\bullet$ is supported on $P(\mathcal F_\bullet, B)$. 
In general, $P(\mathcal F_\bullet, B)$ is not an hcw-poset.      
\end{remarks} 

\begin{example}\label{E:ToricExample}
Consider the surjective homomorphism of abelian groups 
$f\: \mathbb{Z}^2\rightarrow \mathbb{Z}$ given by 
$f(e_1)=2$ and $f(e_2)=3$, and let $L=\ker f$. Thus $L$ 
is the lattice in $\mathbb Z^2$ generated by $\ell=(3,-2)=3e_1-2e_2$, 
and we will identify via the map $f$ the quotient $\mathbb Z^2/L$ with 
$\mathbb Z$. Under this identification, the affine semigroup $Q$ is 
the numeric semigroup in $\mathbb Z$ generated by $2$ and $3$, and the 
corresponding toric ring $S=\Bbbk[Q]\cong\Bbbk[x,y]/(y^2-x^3)$ is  
the coordinate ring 
of the semicubical parabola.  
The induced $\mathbb Z^2/L = \mathbb Z\,$-grading on the polynomial ring 
$R=\Bbbk[x,y]$ has $x$ of degree $2$ and $y$ of degree $3$, and a 
minimal free $(\mathbb Z^2/L)$-graded resolution $\mathcal F_\bullet$ of 
$S$ is 
\[
0 \lla R \xleftarrow[d_1]{\ y^2-x^3} R \lla 0.  
\]
Thus $\mathcal F_\bullet$ has a homogeneous $R$-basis $B=\{b_0, b_1\}$ 
where $b_0$ is a degree $0$ basis element for the free module 
in homological dimension $0$, and $b_1$ is a  basis element 
of degree $6$ for the free module in homological dimension $1$.    
Let $\nu\: \mathbb Z^2/L=\mathbb Z\rightarrow \mathbb Z^2$ be  the 
splitting given by $\nu(1)= 2e_1-e_2=(2,-1)$. This makes $B$ a homogeneous 
basis of $\tau\mathcal F_\bullet$ over $R[L]$ with $\mathbb Z^2$-degrees 
for its elements $\deg(b_0)=\nu(0)=(0,0)$ and $\deg(b_1)=\nu(6)=(12,-6)$.  
Since \ $d_1(b_1)=y^2b_0-x^3b_0$, \ 
Remark~\ref{R:toric-remarks}(d) yields for every $k\in\mathbb Z$ that  
\begin{align*}
(\tau d_1)\bigl((k\ell) b_1\bigr)\ 
&= \ y^2\bigl(\lambda_{(0,2)} + k\ell\bigr)b_0 \ - 
   \ x^3\bigl(\lambda_{(3,0)} + k\ell\bigr)b_0            \\ 
&= \ y^2\bigl(\nu(6)-(0,2)  + k\ell\bigr)b_0 \ - 
   \ x^3\bigl(\nu(6)-(3,0)  + k\ell\bigr)b_0             \\ 
&= \ y^2\bigl((4+k)\ell\bigr)b_0\ - \ x^3\bigl((3+k)\ell\bigr)b_0. 
\end{align*}
Therefore $L\times B$ is a homogeneous $R$-basis with 
minimal boundary support for the $\mathbb Z^2$-graded minimal 
resolution $\tau\mathcal F_\bullet$, 
the toric incidence poset $P(\mathcal F_\bullet, B)$ supporting 
$\mathcal F_\bullet$ is 
the poset structure on the infinite set $L\times B$ with 
Hasse diagram 
\begin{center}
\begin{tikzpicture}[scale=0.9, vertices/.style={draw, fill=black,
                            circle, inner sep=1pt}]
              \node [vertices, label=below:{$(5\ell)b_0$}]            (2) 
                                                  at (4.25, -1.5){};
              \node [label=left:{$\ldots$}]                           (3) 
                                                  at (-2.5, -1.5){};
              \node [vertices, label=below:{$(2\ell)b_0$}]            (4) 
                                                  at (-.25, -1.5){};
              \node [vertices, label=below:{$(3\ell)b_0$}]            (5) 
                                                  at (1.25, -1.5){};
              \node [vertices, label=below:{$(4\ell)b_0$}]            (6) 
                                                  at (2.75, -1.5){};
              \node [vertices, label=below:{$\ell\, b_0$}]            (1) 
                                                  at (-1.75, -1.5){};
	      \node [label=right:{$\ldots$}]                          (15) 
                                                  at (2.75, 0){};
	      \node [label=left:{$\ldots$}]                           (14) 
                                                  at (-4.75, 0){};
              \node [vertices, label=above:{$\smash[b]{(2\ell)b_1}$}] (7) 
                                                  at (-.25, 1.5){};
              \node [vertices, label=above:{$\smash[b]{(-\ell)b_1}$}] (8) 
                                                  at (-4.75, 1.5){};
              \node [vertices, label=above:{$\smash[b]{b_1}$}]        (9) 
                                                  at (-3.25, 1.5){};
              \node [vertices, label=above:{$\smash[b]{\ell\, b_1}$}] (10) 
                                                  at (-1.75, 1.5){};
              \node [label=right:{$\ldots$}]                          (11) 
                                                  at (.5, 1.5){};
              \node [vertices, label=above:{$\smash[b]{(-2\ell)b_1}$}] (12) 
                                                  at (-6.25, 1.5){};

	   \foreach \to/\from in {14/1, 1/12, 12/4, 4/8, 8/5, 5/9, 
                                   9/6, 6/10, 10/2, 2/7, 7/15} 
	   \draw [-] (\to)--(\from);
	   \node [label=right:{$\ldots$}] at (5, -1.5){};
	   \node [label=left:{$\ldots$}] at (-7, 1.5){};
\end{tikzpicture}
\end{center}
and it is straightforward to see that in this particular case this 
is also a hcw-poset. 
\end{example}

\section{Artinianizations and Alexander duality}\label{S:AAD}

Here we show how to translate cellular resolution techniques from 
\cite[Sections 5.3 and 5.4]{MiSt} to the setting of 
resolutions supported on posets.

In what follows, $R=\Bbbk[x_1,\dots,x_m]$ is a polynomial ring over 
a field $\Bbbk$, with the standard $\mathbb Z^m$-grading. 
As in \cite{MiSt}, 
for $\alpha=(a_1,\dots, a_m)\in\mathbb Z^m$ we set 
$\m^\alpha=(x_1^{a_1},\dots, x_m^{a_m})$, and, when   
a monomial ideal $I$ is generated 
in degrees $\le\alpha$, 
we write $I^{[\alpha]}$ for the corresponding 
Alexander dual of the ideal $I$.   
We refer the reader to the excellent exposition in \cite{MiSt} for 
the definition and basic properties of Alexander duality.

For a chain complex 
$\mathcal G_\bullet=(G_n,g_n)$ of finite free $\mathbb Z^m$-graded $R$-modules 
we set $\mathcal G_\bullet^*=\Hom_R\bigl(\mathcal G_\bullet, \ R\bigr)$, 
in particular $G^*_{-n}=\Hom_R(G_n, R)$. 
Also, for $\gamma\in\mathbb Z^m$ we set 
$\mathcal G_\bullet(\gamma)=\mathcal G_\bullet\otimes_R R(\gamma)$, and 
we write $(\mathcal G_\bullet)_{\le\gamma}$ 
for the subcomplex obtained by choosing homogeneous bases for the 
free modules of $\mathcal G_\bullet$ and taking the free submodules generated 
by those basis elements that have degrees  $\le\gamma$. 
It is clear that this does not depend on the choice of a homogeneous 
basis.  

Our last goal in this paper is to give a new proof
for the following consequence of the
fundamental result of Miller~\cite[Theorem 4.5]{Mi}
on Alexander duality:  

\begin{theorem}\label{T:main-theorem-3} 
Let $\alpha=(a_1,\dots, a_m)\in\mathbb Z^m$ and $I$ 
be a monomial ideal generated in degrees $\le\alpha$.  
Let $\mathcal F_\bullet$ be a 
minimal free resolution of $R\big/(I+\m^{\alpha+\mbf{1}})$ over $R$. 
For any element 
$\gamma=(c_1,\dots, c_m)\in\mathbb Z^m$ let 
$\widetilde\gamma=(\tilde c_1,\dots, \tilde c_m)$ be obtained 
by setting for each $i$ that $\tilde c_i=0$ if $c_i\ge a_i+1$, 
and $\tilde c_i=c_i$ otherwise. 
Then: 
\begin{enumerate} 
\item 
$
\Ext_R^m\bigl(R\big/I+\m^{\alpha+\mbf{1}}, \ R\bigr)(-\alpha-\mbf{1})
\, \cong\,  
I^{[\alpha]}\big/ I^{[\alpha]}\cap\m^{\alpha+\mbf{1}} 
$ 
as $\mathbb Z^m$-graded $R$-modules; 

\item
The chain complex 
$\mathcal F_\bullet^*(-\alpha-\mbf{1})[-m]_{\le\alpha}$ is a 
minimal $\mathbb Z^m$-graded free resolution of $I^{[\alpha]}$ over $R$;  

\item
The $\mathbb Z^m$-graded Betti numbers 
$\beta_{m,\gamma}$ 
of \ $R\big/(I+\m^{\alpha+\mbf{1}})$ over $R$ 
in the top homological dimension $m$ are either $0$ or $1$, and 
the intersection 
\[
\bigcap_{\beta_{m,\gamma}=1}\m^{\widetilde{\gamma}}
\]
is an irredundant irreducible decomposition of $I$. 
\end{enumerate} 
\end{theorem}

\begin{example}\label{E:main-theorem-3}
To illustrate the claims of the theorem, 
consider the monomial ideal \ 
$I=(yz, xz, xy)$ \ in $R=\Bbbk[x,y,z]$, and 
take $\alpha=\mathbf{1}=(1,1,1)$. We then have \   
$\mathfrak{m}^{[\alpha+\mathbf{1}]}=(x^2,y^2,z^2)$, \ and \ 
$I+\mathfrak{m}^{[\alpha+\mathbf{1}]}=(x^2, xy, y^2, xz, yz, z^2)$, \ 
and  we note that \ $I=I^{[\mathbf{1}]}$, \ i.e. $I$  
is self-dual. 
A minimal free resolution $\mathcal F_\bullet$ 
for \ $R/(I + \mathfrak{m}^{[\alpha+\mathbf{1}]})$, \   
as produced by Macaulay2~\cite{M2}, is  
{\small
\begin{equation*}
   \begin{CD} 
    {0\la R} 
        @< {
     \begin{bsmallmatrix*}[r]
     x^2 &    xy&     y^2&     xz&   yz&    z^2 
     \end{bsmallmatrix*}
           } < d_1 < 
    {R^{6}} 
	@< {
     \begin{bsmallmatrix*}[r]
     {-y}&    0&      {-z}&      0&      0&      0&      0&      0\\
      x&      {-y}&      0&      {-z}&      0&      0&      0&      0\\
      0&      x&      0&      0&      0&      {-z}&      0&      0\\
      0&      0&      x&      y&      {-y}&      0&      {-z}&      0\\
      0&      0&      0&      0&      x&      y&      0&      {-z}\\
      0&      0&      0&      0&      0&      0&      x&      y\\
     \end{bsmallmatrix*}
     	} <d_2< 
    {R^{8}} 
	@< 
	{
	\begin{bsmallmatrix*}[r]
	z&      0&      0\\
      0&      z&      0\\
      {-y}&      0&      0\\
      x&      {-y}&      0\\
      0&      {-y}&      z\\
      0&      x&      0\\
      0&      0&      {-y}\\
      0&      0&      x\\                       
      \end{bsmallmatrix*}
     } <d_3< 
    {R^3\la 0.}
   \end{CD}
\end{equation*}
}
\negthinspace
\negthinspace
The three basis elements in the top homological dimension  
$m=3$ have degrees  
\[
\gamma_1 =(2,1,1), \quad 
\gamma_2 =(1,2,1), \quad 
\gamma_3 =(1,1,2).
\]
Thus: 
\begin{itemize}
\item[(1)] 
Dualizing we see that \ 
$\Ext_R^3\bigl(R\big/I+\m^{\alpha+\mbf{1}}, \ R\bigr)$ \ has 
minimal generators \ $e_1, e_2, e_3$ \ over $R$ of degrees 
$(-2,-1,-1)$, $(-1,-2,-1)$, and $(-1,-1,-2)$, respectively, 
subject to the relations 
\begin{gather*}
ye_1=ze_1=xe_2=ze_2=xe_3=ye_3=0, \quad\text{ and }\quad \\   
xe_1=ye_2=ze_3.
\end{gather*} 
After a degree twist by $-\alpha-\mathbf{1}=(-2,-2,-2)$,  
these become precisely  the degrees and the relations that the 
minimal generators \ $yz, xz, xy$ \ of \ $I^{[\mathbf{1}]}=I$ \ 
satisfy  modulo \ 
%
$
I^{[\mathbf{1}]}\cap\mathfrak{m}^{[\alpha+\mathbf{1}]} = 
(x^2y, xy^2, xz^2, y^2z, yz^2, xz^2). 
$

\item[(2)] 
The chain complex \ 
$
\mathcal F_\bullet^*(-\alpha-\mbf{1})[-m]_{\le\alpha} \ = \  
\mathcal F_\bullet^*(-2,-2,-2)[-3]_{\le(1,1,1)}
$ 
\ is  
\begin{equation*}
   \begin{CD} 
   {0\lra 0} 
        @> 0 > d_1^*[-3] >  
   { 0 } 
	@> 0 > d_2^*[-3] > 
    {R^{2}} 
	@>  
	{
	\begin{bsmallmatrix*}[r]
       x&      0  \\
    {-y}&   {-y}  \\
       0&      z                            
        \end{bsmallmatrix*}
        } > d_3^*[-3] > 
    {R^3\lra 0,}
   \end{CD}
\end{equation*}
and this is clearly a minimal free resolution of \ $I^{[\mathbf{1}]}=I$. 

\item[(3)] 
The only nonzero multigraded Betti numbers in dimension $3$ are  
\[
\beta_{3,\gamma_1}=\beta_{3,\gamma_2}=\beta_{3,\gamma_3}=1.  
\]
Applying the tilde operation to the corresponding multidegrees we 
obtain  
\[
\tilde{\gamma}_1=(0,1,1), \ \tilde{\gamma}_2=(1,0,1), \  
\tilde{\gamma}_3=(1,1,0),
\]
and we see that  
\[
\mathfrak{m}^{\tilde{\gamma}_1}\cap
\mathfrak{m}^{\tilde{\gamma}_2}\cap
\mathfrak{m}^{\tilde{\gamma}_3}
=(y,z)\cap(x,z)\cap(x,y)
\] 
is precisely the irredudant, irreducible decomposition of $I$. 
\end{itemize}
\end{example} 

We postpone our proof of Theorem~\ref{T:main-theorem-3} 
till the end of the section. The key 
new ingredient there is the existence of a hcw-poset that 
supports $\mathcal F_\bullet$. The rest of the proof uses notions 
and facts about cellular resolutions translated essentially verbatim 
to the setting of resolutions supported on posets, and our next 
task is to state these. 

For the rest of this paper 
{\bf $P$ always denotes an hcw-poset.}
  
We begin by introducing the 
conic chain complex of a pair, the notion of a 
co-labeled poset, and the relevant homogenizations. 
The reader should compare these with 
\cite[Definitions 5.30, 5.31, and 5.33]{MiSt}. 

\begin{definition}
Let $P$ be a poset structure on a finite set $B$, 
let $L$ be a lower set in $P$, let 
$\deg\: B\lra \mathbb Z^m$ be a function, let $\Bbbk$ be a field, and 
let $\mathcal F_\bullet$ be a finite free $\mathbb Z^m$-graded 
chain complex over the polynomial ring $R=\Bbbk[x_1,\dots,x_m]$.

(a) 
The \emph{conic chain complex} of the pair $(P,L)$ over $\Bbbk$ is 
the quotient chain complex 
\[
\mathcal C_{\bullet}(P,L,\Bbbk)=
\mathcal C_{\bullet}(P,\Bbbk)/\mathcal C_{\bullet}(L,\Bbbk).   
\]
Thus  
$
\mathcal C_n(P,L,\Bbbk)=
\bigoplus_{a\in P\setminus L,\ d(a)=n}\HH_n(\Delta_{\le a},\Delta_{<a},\Bbbk) 
$ 
in homological degree $n$, \linebreak hence the  $P$-grading on 
$\mathcal C_\bullet(P,\Bbbk)$ induces a natural $P$-grading on 
$\mathcal C_\bullet(P,L,\Bbbk)$. 

(b) 
We write ${\widetilde{\mathcal C}}^*_\bullet(P,\Bbbk)$ and 
$\mathcal C^*_\bullet(P,L,\Bbbk)$ 
for the shifted dual chain complexes   
\begin{align*}
{\widetilde{\mathcal C}}^*_\bullet(P,\Bbbk) &= 
\Hom_\Bbbk\bigl(\widetilde{\mathcal C}_\bullet(P,\Bbbk),\ \Bbbk\bigr)[d] \\ 
\mathcal C^{*}_\bullet(P,L,\Bbbk) &=
\Hom_\Bbbk\bigl(\mathcal C_{\bullet}(P,L,\Bbbk),\ \Bbbk\bigr)[d],  
\end{align*}
where $d=\dim P$. Thus 
$
\mathcal C_n^*(P,L,\Bbbk)=
\Hom_\Bbbk\bigl(\mathcal C_{d-n}(P,L,\Bbbk),\ \Bbbk\bigr)$ 
in homological degree $n$, hence the $P$-grading on 
$\mathcal C_\bullet(P,L,\Bbbk)$ induces a natural 
$P^{op}$-grading on $\mathcal C_\bullet^*(P,L,\Bbbk)$. 
Similarly, 
$
\widetilde{\mathcal C}_n^*(P, \Bbbk)=
\Hom_\Bbbk\bigl(\widetilde{\mathcal C}_{d-n}(P, \Bbbk),\ \Bbbk\bigr)$ 
in homological degree $n$, and the $\check P$-grading on 
$\widetilde{\mathcal C}_\bullet(P, \Bbbk)$ induces a natural 
$\widehat{P^{op}}$-grading on $\widetilde{\mathcal C}_\bullet^*(P, \Bbbk)$. 


(c) Let $X=(P,\deg)$, and suppose that $(P^{op},\deg)$ is a labeled poset. 
We say that $X$ is a \emph{co-labeled poset}. 
We say that the chain complex $\mathcal F_\bullet$ is 
\emph{co-supported on $Y=(P,L,\deg)$} if it is isomorphic to the 
homogenization $\mathcal C_\bullet^*(Y,\Bbbk)$ of the 
$P^{op}$-graded chain complex 
$\mathcal C_\bullet^*(P,L,\Bbbk)$ with respect to the labeled poset 
$(P^{op},\deg)$. 
We also write $\widetilde{\mathcal C}^*_\bullet(X, \Bbbk)$ for the 
homogenization of the $\widehat{P^{op}}$-graded chain complex  
$\widetilde{\mathcal C}^*_\bullet(P, \Bbbk)$ with respect to the labeled 
poset $\bigl(\widehat{P^{op}}, \widehat{\deg}\bigr)$. 
Here $\widehat{\deg}$ is the labeling on $\widehat{P^{op}}$ that agrees 
with $\deg$ on $P^{op}$, and has $\widehat{\deg}(\hat 1)$ as the 
coordinatewise maximum in $\mathbb Z^m$ of the elements in the set 
$\{\deg(a)\mid a\in P^{op}\}$. 
%
\end{definition}

We recall that if a labeled hcw-poset $P$ supports a free resolution 
$\mathcal F_\bullet$ then each element of $P$ corresponds to a (unique 
up to associates) basis element of $\mathcal C_\bullet(P,\Bbbk)$ and 
hence to a (unique up to associates) basis element of $\mathcal F_\bullet$; 
much in the same way as the cells of a regular CW-complex $X$ provide a 
basis for a cellular resolution supported on $X$. This simple observation 
is the main (and often the only) fact that is needed to convert a statement 
about cellular resolutions into a statement about resolutions supported 
on hcw-posets, and it is what we are referring to when we use the phrase 
\emph{mutatis mutandis} in the sequel. In the next six results we 
demonstrate this conversion process explicitly by reformulating 
cellular resolution statements and their proofs from \cite{MiSt}. 

The reader should compare the following lemma and its proof with 
\cite[Lemma 5.35 and its proof]{MiSt}.  

\begin{lemma}\label{modules-to-ideals}
Let $\alpha\in\mathbb Z^m$,  
let $I$ be a monomial ideal generated in degrees 
$\le \alpha$, let $X=(P,\deg)$ be a co-labeled poset, let 
$L$ be a lower set in $P$, and let $Y=(P,L,\deg)$. 
\begin{enumerate} 
\item 
If the complex \ $\widetilde{\mathcal C}^*_\bullet(X,\Bbbk)$ 
is a minimal free resolution 
of $I\big/ I\cap\m^{\alpha+\mbf{1}}$, then 
$Z=(P,\ P_{\deg\not\le\alpha},\ \deg)$ co-supports 
a minimal free resolution of $I$.

\item 
If the complex \ 
$\mathcal C^*_\bullet(Y,\Bbbk)$ is a (minimal) free resolution 
of $I\big/ I\cap\m^{\alpha+\mbf{1}}$, then 
$W=(P,\ L\cup P_{\deg\not\le\alpha},\ \deg)$ 
co-supports a (minimal) free resolution of $I$.
\end{enumerate}    
\end{lemma}

\begin{proof} 
(1)  \ 
The exact sequence of $\mathbb Z^m$-graded 
Cohen-Macaulay $R$-modules 
\[
0\lra I\big/I\cap\m^{\alpha+\mbf{1}} \lra 
R\big/\m^{\alpha+\mbf{1}} \lra R\big/(I+\m^{\alpha+\mbf{1}}) \lra 0
\]
yields that 
\[
R/\m^{\alpha+\mbf{1}}(\alpha+\mbf{1})\cong 
\Ext^m_R\bigl(R/\m^{\alpha+\mbf{1}}, R\bigr) \ra  
\Ext^m_R\bigl(I\big/I\cap\m^{\alpha+\mbf{1}}, R\bigr) \ra 0
\] 
is exact, hence the single basis element in homological degree $m$ of 
$\widetilde{\mathcal C}^*_\bullet(X,\Bbbk)$ has $\mathbb Z^m$-degree 
$\alpha+\mbf{1}$.   
%
Therefore the degree 
$\beta\in\mathbb Z^m$ graded strand of 
$\mathcal C^*_\bullet(Z, \Bbbk)$ is precisely 
the degree $\alpha\wedge\beta$ graded strand of 
$\widetilde{\mathcal C}^*_\bullet(X,\Bbbk)$, in particular it is acyclic. Thus 
$\mathcal C^*_\bullet(Z, \Bbbk)$ is a minimal free resolution. Comparing 
generators and relations as in \cite[Proof of Lemma~5.35]{MiSt} we see 
that it is a resolution of $I$. 

(2) \ 
Note that the degree 
$\beta\in\mathbb Z^m$ graded strand of 
$\mathcal C^*_\bullet(W, \Bbbk)$ is precisely 
the degree $\alpha\wedge\beta$ graded strand of 
$\mathcal C^*_\bullet(Y,\Bbbk)$ and proceed as in the proof of part (1). 
\end{proof}

An \emph{acyclic cover over $\Bbbk$} of a poset $P$ is a collection 
of lower sets $\mathfrak L=\{L_1,\dots, L_n\}$ such that 
$P=L_1\cup\dots\cup L_n$ and 
for any subset $U\subseteq\{1,\dots, n\}$ either the intersection 
$L_U=\bigcap_{i\in U}L_i$ is empty or 
$\smash[t]{\widetilde{\mathcal C}_\bullet(L_U,\Bbbk)}$ is exact.   
The \emph{nerve} of  $\mathfrak L$ is the abstract simplicial complex 
$N(\mathfrak L)$ on the set of vertices $\{1,\dots, n\}$ whose  
nonempty faces are the subsets $U\subseteq\{1,\dots, n\}$ such that $L_U$ is 
not empty. 

The reader should compare the next lemma with 
\cite[Lemma 5.36]{MiSt}.

\begin{lemma}[Nerve Lemma] 
If $\mathfrak L$ is an acyclic cover over $\Bbbk$ of a poset $P$ 
then 
$
\HH_i\widetilde{\mathcal C}_\bullet(P, \Bbbk)\cong 
\RH_i\bigl(N(\mathfrak L), \Bbbk\bigr)
$
for every $i$. 
\end{lemma}

\begin{proof}
For $k=0,\dots, n-1$ set 
\[
C_{\bullet,k}=\bigoplus_{\substack{|U|=k+1 \\ U\subset\{1,\dots,n\}}}
\mathcal C_\bullet(L_U,\Bbbk). 
\] 
For $1\le k\le n-1$ let 
\[
\phi_k\: C_{\bullet,k}\lra C_{\bullet,k-1}
\] 
be the morphism of complexes whose component mapping 
$\mathcal C_\bullet(L_U,\Bbbk)$ to $\mathcal C_\bullet(L_V,\Bbbk)$ is 
$0$ if $V\not\subset U$ and is $\sign\sigma$ times the natural inclusion 
$\mathcal C_\bullet(L_U,\Bbbk)\subseteq\mathcal C_\bullet(L_V,\Bbbk)$ otherwise. 
(Here we identify each subset of $\{1,\dots,n\}$ with the sequence of its  
elements arranged in increasing order, and $\sigma$ is the permutation 
that arranges the elements of 
the sequence $(U\setminus V,\ V)$ in increasing order.)  
It is straightforward to check that we obtain a chain complex of 
acyclic chain complexes 
\begin{equation}\label{E:double-complex}
0\lra C_{\bullet,n-1} 
\xrightarrow{\phi_{n-1}}\dots 
\xrightarrow{\phi_{k+1}}  
C_{\bullet,k} \xrightarrow{\phi_k} C_{\bullet, k-1}
\xrightarrow{\phi_{k-1}} \dots 
\xrightarrow{\phi_1}  
C_{\bullet, 0}\lra 0.  
\end{equation}
Since for any lower sets $K$ and $M$ in $P$ we have 
\begin{gather*}
\mathcal C_\bullet(K\cap M,\Bbbk)=
\mathcal C_\bullet(K,\Bbbk) \cap \mathcal C_\bullet(M,\Bbbk), \\
\mathcal C_\bullet(K\cup M,\Bbbk)=
\mathcal C_\bullet(K,\Bbbk) + \mathcal C_\bullet(M,\Bbbk)
\end{gather*}
as subcomplexes of $\mathcal C_\bullet(P,\Bbbk)$, we see  that
$\coker\phi_1\cong\mathcal C_\bullet(P,\Bbbk)$. Furthermore, 
it is immediate from the definitions that in each 
homological degree $d$ the induced chain complex 
\begin{equation*}
0\lra C_{d,n-1} 
\xrightarrow{\phi_{n-1}}\dots 
\xrightarrow{\phi_{k+1}}  
C_{d,k} \xrightarrow{\phi_k} C_{d, k-1}
\xrightarrow{\phi_{k-1}} \dots 
\xrightarrow{\phi_1}  
C_{d, 0}\lra 0.  
\end{equation*}
is isomorphic to 
\[
\bigoplus_{\substack{a\in P \\ d(a)=d}}
\HH_d(\Delta_{\le a},\Delta_{<a},\Bbbk)\otimes C_\bullet(\Delta_{U_a},\Bbbk), 
\] 
where $U_a=\{i\mid a\in L_i\}$ and $\Delta_{U_a}$ is the full 
simplex on the set of vertices $U_a$. It follows that the chain complex 
\eqref{E:double-complex} is acyclic, hence a resolution of 
$\mathcal C_\bullet(P,\Bbbk)$. Now the claim of the Nerve Lemma 
follows from comparing the spectral sequences of the 
two canonical filtrations of the 
total complex associated with \eqref{E:double-complex}:   
the first spectral sequence collapses on the second page to the 
homology of $\mathcal C_\bullet(P,\Bbbk)$, while the second spectral sequence,
due to the acyclicity of the cover $\mathfrak L$,  
collapses on the second page to the homology of the simplicial 
chain complex of $N(\mathfrak L)$ over $\Bbbk$. 
\end{proof} 

Compare the following proposition and its proof with 
\cite[Proposition 5.37 and its proof]{MiSt}.

\begin{proposition}\label{P:resolving-duals}
Let $I$ be a monomial ideal generated in degrees 
$\le\alpha\in\mathbb Z^m$. Let $X=(P, \deg)$ be a labeled poset that 
supports a minimal free resolution
%
of  $I+\m^{\alpha+\mbf{1}}$. 

Then $\alpha+\mbf{1}-X=(P,\ \alpha+\mbf{1}-\deg)$  
is a co-labeled poset, 
$\widetilde{\mathcal C}^*_\bullet(\alpha+\mbf{1}-X,\ \Bbbk)$ is a minimal 
free resolution of\/  
$I^{[\alpha]} \big/ I^{[\alpha]} \cap \m^{\alpha+\mbf{1}}$, and 
$(P,\ P_{\deg\not\ge\mbf{1}},\ \alpha+\mbf{1}-\deg)$ co-supports a 
minimal free resolution of\/ $I^{[\alpha]}$. 
\end{proposition}

\begin{proof} 
By Lemma~\ref{modules-to-ideals} it is enough to show that 
$\widetilde{\mathcal C}^*_\bullet(\alpha+\mbf{1}-X, \Bbbk)$ is 
a minimal free resolution of 
$I^{[\alpha]} \big/ I^{[\alpha]} \cap \m^{\alpha+\mbf{1}}$. 
Since $I+\m^{\alpha+\mbf{1}}$ is Cohen-Macaulay of codimension $m$, we have
according to our definitions that 
\[
\widetilde{\mathcal C}^*_\bullet(\alpha+\mbf{1}-X, \Bbbk)=
\Hom_R\bigl(
\widetilde{\mathcal C}_\bullet(X, \Bbbk),R\bigr)[m-1](-\alpha-\mbf{1}), 
\]
thus it 
is a minimal 
free resolution of 
$M=\Ext_R^m\bigl(R\big/I+\m^{\alpha+\mbf{1}},\ R\bigr)(-\alpha-\mbf{1})$. 
Furthermore, the exact sequence of Cohen-Macaulay modules 
\[
0\lra I+\m^{\alpha+\mbf{1}}\big/\m^{\alpha+\mbf{1}} \lra 
R\big/\m^{\alpha+\mbf{1}} \lra R\big/I+\m^{\alpha+\mbf{1}} \lra 0
\]
yields that 
$
0\lra M \lra \Ext^m_R\bigl(R/\m^{\alpha+\mbf{1}}, R\bigr)(-\alpha-\mbf{1})= 
R/\m^{\alpha+\mbf{1}}$ 
is exact, hence it is enough to show that the degree $\beta$ component of 
$M$ agrees with the degree $\beta$ component of 
$I^{[\alpha]} \big/ I^{[\alpha]} \cap \m^{\alpha+\mbf{1}}$ for every 
$\beta\le\alpha$. 

Note that the 
degree $\beta\in\mathbb Z^m$ graded 
strand of 
$\widetilde{\mathcal C}_\bullet^*(\alpha+\mbf{1}-X,\Bbbk)$ is 
canonically isomorphic to  
$\mathcal C_\bullet^*(P, P_{\deg\not\ge\alpha+\mbf{1}-\beta}, \Bbbk)$ 
whenever $\beta\le\alpha$. 
Since $\mathcal{\widetilde C}_\bullet(P,\Bbbk)$ is exact, we also have that 
\[
\HH_i
\mathcal C_\bullet^*(P, P_{\deg\not\ge\alpha+\mbf{1}-\beta}, \Bbbk)  
\cong 
\HH_{i+1-e}
\mathcal{\widetilde C}_\bullet^*(P_{\deg\not\ge\alpha+\mbf{1}-\beta}, \Bbbk)
\]  
where $e=\dim P - \dim P_{\deg\not\ge\alpha+\mbf{1}-\beta}=(m-1) - d_{\beta}$. 
Since for any $n$ we have 
$
\HH_n\mathcal{\widetilde C}_\bullet^*(P_{\deg\not\ge\alpha+\mbf{1}-\beta},\Bbbk) 
\cong 
\HH_{d_\beta-n}\mathcal{\widetilde C}_\bullet(P_{\deg\not\ge\alpha +\mbf{1}-\beta},\Bbbk),
$
we obtain that 
\[
\HH_i
\mathcal C_\bullet^*(P, P_{\deg\not\ge\alpha+\mbf{1}-\beta}, \Bbbk)  
\cong 
\HH_{m-2-i}
\mathcal{\widetilde C}_\bullet(P_{\deg\not\ge\alpha+\mbf{1}-\beta}, \Bbbk). 
\]  
Let $S=\{1,\dots, m\}$. For each $i\in S$ 
let $L_i=\{a\in P\mid (\deg a)_i< \alpha_i +1 -\beta_i\}$. Thus 
$L_i=P_{\deg\le \alpha+\mbf{1}-(\beta_i+1)e_i}$ where $e_i$ is the $i$th standard basis 
element of $\mathbb Z^m$. In particular $L_i$ is a lower set  in $P$, 
and  $P_{\deg\not\ge\alpha+\mbf{1}-\beta}=\bigcup_{i=1}^m L_i$. 
Furthermore,  for any $U\subseteq S$ we obtain the lower set 
$L_U=\bigcap_{i\in U} L_i = P_{\deg\le\alpha +\mbf{1} -\sum_{i\in U}(\beta_i+1)e_i}$, 
hence  
\smash[t]{$\widetilde{\mathcal C}_\bullet(L_U,\Bbbk)$} is exact by 
Proposition~\ref{acyclicity-criterion} whenever $L_U$ is not empty. 
Therefore the collection $\mathfrak L=\{L_1,\dots, L_m\}$
is an acyclic cover over $\Bbbk$ of $P_{\deg\not\ge\alpha+\mbf{1}-\beta}$. 

Now we analyze the nerve of $\mathfrak L$. 
For each 
$j\in S$ the lower set $L_{S\setminus j}$ is not empty as $P$ 
contains an element of degree $(\alpha_j+1)e_j$; therefore 
$N(\mathfrak L)$ 
is either the full simplex on $\{1,\dots, m\}$ or its boundary. It follows 
that 
$
\HH_i 
\mathcal C_\bullet^*(P, P_{\deg\not\ge\alpha+\mbf{1}-\beta}, \Bbbk)  
$  
is possibly nonzero only when $N(\mathfrak L)$ is an $(m-2)$-dimensional 
sphere. In that case  
\[
\HH_i
\mathcal C_\bullet^*(P, P_{\deg\not\ge\alpha+\mbf{1}-\beta}, \Bbbk)  
\cong 
\begin{cases}
\Bbbk &\text{ if } i= 0, \\ 
0     &\text{ otherwise,} 
\end{cases}
\]
by the Nerve Lemma. Note also that this occurs exactly when $P$ does 
not have any elements of degree $\le\alpha-\beta$, which (since $X$ 
supports the \emph{minimal} free resolution of $I+\m^{\alpha+\mbf{1}}$) is 
equivalent to $x^{\alpha-\beta}\notin I$, which is equivalent to 
$x^\beta\in I^{[\alpha]}$ by \cite[Proposition 5.23]{MiSt}. Therefore 
the degree $\beta$ homogeneous component of $M$ 
agrees with the degree $\beta$ homogeneous component of 
$
I^{[\alpha]}+\m^{\alpha+\mbf{1}}\big/\m^{\alpha+\mbf{1}}\cong 
I^{[\alpha]}\big/ I^{[\alpha]}\cap\m^{\alpha+\mbf{1}} 
$
whenever $\beta\le\alpha$. 
This completes the proof of the proposition. 
\end{proof}

\begin{lemma}
If a labeled poset $(P, \deg)$ supports a minimal free resolution of a 
Cohen-Macaulay monomial ideal $I$ of codimension $g$, then all 
maximal elements of $P$ have dimension $g-1$.   
\end{lemma}

\begin{proof} 
The proof is \emph{mutatis mutandis} that of 
\cite[Corollary~5.39]{MiSt}. 
\end{proof}

\begin{lemma} 
Let $a$ be an element of a labeled poset $(P, \deg)$ supporting 
a minimal resolution of a monomial ideal $I$ of $R$ with $R/I$ 
Artinian. If the $i$'th coordinate $\deg(a)_i$ of\/ $\deg(a)$ is nonzero, 
then $\deg(a)_i=\deg(a')_i$ for some element $a'\ge a$ of $P$ that 
is maximal among the elements of\/ $P$ whose degrees have the same 
support as $\deg(a)$. Any such element\/ $a'$ has dimension  
$d(a')=\bigl|\supp\bigl(\deg(a)\bigr)\bigr| - 1$.   
\end{lemma}

\begin{proof} 
The proof is \emph{mutatis mutandis} that of 
\cite[Proposition~5.40]{MiSt}. 
\end{proof}

\begin{proposition}\label{P:irredundant-decomposition} 
Let $\alpha=(a_1,\dots, a_m)$ and consider a monomial 
ideal $I$ that is generated in degrees 
$\le\alpha$.  Let $X=(P,\deg)$ be a labeled poset 
supporting a minimal free resolution of $I+\m^{\alpha+\mbf{1}}$. 
Let $D$ be the set of 
maximal elements of $P$.  

Then no two elements of $D$ have the same \ $\widetilde{\deg}$, 
and the intersection \ 
\[
\bigcap_{a\in D}\m^{\widetilde{\deg(a)}}
\] 
is an irredundant irreducible decomposition of the monomial ideal $I$.    
\end{proposition}

\begin{proof}
The proof is \emph{mutatis mutandis} that of 
\cite[Theorem~5.42]{MiSt}
\end{proof}


We hope by now the reader is convinced that virtually every 
statement about cellular resolutions can be essentially verbatim 
translated into a statement about resolutions supported on hcw-posets. 
At this point we are ready to leverage the fact that every minimal free resolution is supported on an hcw-poset. 

\begin{proof}[Proof of Theorem \ref{T:main-theorem-3}] 
By Theorem~\ref{T:cl-supports-I} there exists a labeled poset $X=(P,\deg)$,  
with $P$ an hcw-poset, that supports a minimal free resolution 
of $I+\m^{\alpha+\mbf{1}}$. Therefore we have 
\[
\widetilde{\mathcal C}^*_\bullet\bigl(\alpha+\mbf{1}-X, \ \Bbbk\bigr) 
\cong 
\mathcal F_\bullet^*(-\alpha-\mbf{1})[-m]. 
\] 
Now (1) and (2) follow from Proposition~\ref{P:resolving-duals},  and (3) 
follows from Proposition~\ref{P:irredundant-decomposition} and 
the fact that $P$ is an hcw-poset. 
\end{proof}


We conclude by remarking that in the proof of 
Theorem~\ref{T:main-theorem-3} we did not use any additional properties 
of the hcw-poset $P$ other than its existence. Thus, even though in 
general the computation of an hcw-poset supporting a given free resolution 
may be complicated, the mere fact that such a poset exists already has  
significant structural implications.

\end{document}